\documentclass[reqno,centertags,a4paper,10pt]{amsart}
\usepackage{amssymb,enumerate}
\usepackage{amsmath}
\usepackage{amsthm}
\usepackage{bm}                       
\usepackage{eucal}
\usepackage{fullpage}

\usepackage[utf8]{inputenc}		

\usepackage{indentfirst}		
\usepackage{color}				
\usepackage{graphicx}			
\usepackage{microtype} 			
\usepackage{amsmath}

\usepackage{hyperref}     

\usepackage[all]{xy}

\usepackage{pgf,tikz}

\usetikzlibrary{arrows}

\newcommand{\R}{\textnormal{I}\!\textnormal{R}}
\newcommand{\N}{\textnormal{I}\!\textnormal{N}}

\newcommand{\secao}[1]{\section{#1}\setcounter{equation}{0}}

\newtheorem{theorem}{Theorem}[section]
\newtheorem{proposition}[theorem]{Proposition}
\newtheorem{remark}[theorem]{Remark}
\newtheorem{lemma}[theorem]{Lemma}

\numberwithin{equation}{section}

\begin{document}
	\title[Sharp well-posedness]{Sharp well-posedness for a coupled system of  mKdV  type equations}
	\author{Xavier Carvajal}
	\address{Instituto de Matem\'atica, Universidade Federal do Rio de Janeiro-UFRJ. Ilha do Fund\~{a}o,
		21945-970. Rio de Janeiro-RJ, Brazil}
	\email{carvajal@im.ufrj.br}
	\author[L. Esquivel]{Liliana Esquivel}
	\address{Gran Sasso Science Institute, CP 67100,  L' Aquila, Italia.}
	\email{liliana.esquivel@gssi.it}
	\author{Raphael Santos}
	\address{Universidade Federal do Rio de Janeiro, Campus Macaé, Brazil}
	\email{raphaelsantos@macae.ufrj.br}
	
	
	\begin{abstract}
		We consider the initial value problem associated to a  system consisting     modified Korteweg-de Vries  type equations 
		\begin{equation*}
			\begin{cases}
				\partial_tv + \partial_x^3v + \partial_x(vw^2) =0,&v(x,0)=\phi(x),\\
				\partial_tw + \alpha\partial_x^3w + \partial_x(v^2w) =0,& w(x,0)=\psi(x),
			\end{cases}
		\end{equation*}
		and prove  the local well-posedness results for   given data in low regularity Sobolev spaces
		$H^{s}(\R)\times H^{k}(\R)$, $s,k> -\frac12$ and $|s-k|\leq 1/2$, for $\alpha\neq 0,1$. Also, we prove that: (I) the solution mapping that takes initial data to the solution fails to be $C^3$ at the origin, when $s<-1/2$ or $k<-1/2$ or $|s-k|>2$; (II) the  trilinear estimates used in the proof of the local well-posedness theorem fail to hold when (a) $s-2k>1$ or $k<-1/2$ (b) $k-2s>1$ or $s<-1/2$; (c) $s=k=-1/2 $; (III) the local well-posedness result is sharp in a sense that  we can not reduce the proof of the trilinear estimates, proving some related bilinear estimates (as in \protect\cite{Tao2001}).
	\end{abstract}
	
	\maketitle

	\secao{Introduction}
	This paper is devoted to the initial value problem (IVP) for the  system of the  modified Korteweg-de Vries (mKdV)-type equations
	\begin{equation}\label{ivp-sy}
		\begin{cases}
			\partial_tv + \partial_x^3v + \partial_x(vw^2) =0,&v(x,0)=\phi(x),\\
			\partial_tw + \alpha\partial_x^3w + \partial_x(v^2w) =0,& w(x,0)=\psi(x),
		\end{cases}
	\end{equation}
	where   $(x, t) \in \R\times \R$; $v=v(x,t)$ and $w= w(x, t)$ are real-valued functions,  and $\alpha\in\R$ is a constant. 
	
	For $\alpha=1$, among a vast class of nonlinear evolution equations, the related system was studied by \protect\cite{AKN1973}, in the context of inverse scattering, showing that this method provides a means of solution of the associated IVP. For existence and estability of solitary waves to the system  \protect\eqref{ivp-sy} we refer the works   \protect\cite{AAM1999} and \protect\cite{Montenegro1995}.

	The well-posedness for the IVP \eqref{ivp-sy} with initial data in the classical Sobolev spaces $H^s(\R)\times H^k(\R)$ was studied by many authors. In 1995, following Kenig, Ponce and Vega \protect\cite{KPV1993},  using smoothing properties of the group, Maximal functions ans Strichartz estimates,  Montenegro  \protect\cite{Montenegro1995} proved that the IVP (\ref{ivp-sy}) with $\alpha =1$ is locally
	well-posed for given data $(\phi, \psi)$ in $H^s(\R)\times H^s(\R)$, $s\geq \frac{1}{4}$. He also proved global well-posedness for given data in $H^s(\R)\times H^s(\R)$, $s\geq 1$, using the conservation laws
	\begin{equation*}\label{con21}
		I_1(v, w) := \int_{\R} (v^2 +w^2)\,dx
		\ \ \textrm{and} \ \ 
		I_2(v, w) := \int_{\R} (v_x^2 + w_x^2 - v^2w^2)\, dx.
	\end{equation*}
	We note that the approach in \protect\cite{KPV1993} implies the local well-posedness for $s\geq 1/4$, when $0<\alpha< 1$. 
	
	In 1999, Alarcon, Angulo and Montenegro  \protect\cite{AAM1999} studied some properties of the solutions for the system of nonlinear evolution equation 
	\begin{equation}
		\left\{
		\begin{array}{c} \label{1}
			\partial_tv+\partial_x^3 v+\partial_x(v^{p}u^{p+1})=0, \ \ u(x,0)=\phi(x)\\
			\\
			\partial_tu+ \partial_x^3 u+\partial_x(v^{p+1}u^p)=0, \ \ v(x,0)=\varphi(x)
		\end{array}\right.
	\end{equation}
	for $p\geq 1$. In this work they proved that \eqref{1} has a family of solitary wave solutions, similar to those found for Korteweg de Vries(KdV)-type equations and that it can be stable or unstable depending on the range of $p$.  We observe the system \eqref{ivp-sy}, with $\alpha=1$, is a special case of \eqref{1} with $p=1$.  In this approach they also uses the smoothing property of the linear group combined with the $L_x^pL_t^q$ Strichartz estimates and maximal function estimates.
	
	In 2001, Tao \protect\cite{Tao2001} shows that the trilinear estimate is valid for $s\geq 1/4$
	\begin{equation}
		\label{Tri-ln}
		\|\partial_x(uvw)\|_{X_{s,b'}}\lesssim \|u\|_{X_{s,b}}\|v\|_{X_{s,b}}\|w\|_{X_{s,b}}
	\end{equation}
	when $X_{s,b}$ is the Bourgain space (see \protect\cite{Bourgain1993}). This leads us to get also the local well-posedness for the system \eqref{ivp-sy} when $\alpha=1$, for $s\geq 1/4$ in the context of Fourier restriction norm method. It is worth noting that the local well-posedness result for the system \eqref{ivp-sy} with $\alpha =1$ is sharp and it can be justified in two different way; first the trilinear estimates fail if $s<1/4$ (see Theorem 1.7 in \protect\cite{KPV1996}). Second, the solution map is not uniformly continuous if $s<1/4$ (see Theorem 1.3 in \cite{KPV2001}). This notion of ill-posedness is a bit strong. 
	For further works in this direction, we refer  \protect\cite{CCT2003}.  
	In 2012, Corcho and Panthee in \protect\cite{CP2011} improves the global result in \cite{Montenegro1995},  getting global well-posedness in $H^s\times H^s$, for $s>1/4$, for $\alpha=1$, see also \protect\cite{Carvajal2006}. 
	
	Recently, in 2019 Carvajal and Panthee  \protect\cite{CP2019} proved local well-posedness in $H^s\times H^s$ for $s>-1/2$, when $\alpha\in (0,\,1)\cup(1,\infty)$. Also, they proved that the key trilinear estimates fails to  hold  and also the solution map is not $C^3$ at the origin, both when $s<-1/2$. Observe that this result also is sharp, considering the scaling argument $s=-1/2$ to the modified KdV equation. 
	
	Many authors studies local well-posedness  for a system with dispersive equations, when the initial data belongs to diferents Sobolev spaces, i.e., in $H^s\times H^k$, $k\neq s$  (see, e.g., Ginibre, Tsutsumi and Velo \protect\cite{GTV1997}).
	In this context,  we prove the following local well-posedness result:
	\begin{theorem}\label{loc-sys}
		Let $\alpha \neq 0,1$, $b>1/2$ and $s, k$ such that $s,k  >-\frac12$ and $|s-k|\leq 1/2$. 
		Then for any $(\phi,\psi)\in H^{s}(\R)\times H^{k}(\R)$,  there exist $\delta = \delta(\|(\phi,\psi)\|_{H^{s}\times H^{k}})$ (with $\delta(\rho)\to \infty$ as $\rho\to 0$) and a unique solution $(v,w)\in X^{\delta}_{s, b}\times X^{\alpha, \delta}_{k, b}$ to the IVP \eqref{ivp-sy} in the time interval $[0, \delta]$. Moreover, the solution satisfies the estimate
		\begin{equation*}\label{est-1s}
			\|(v, w)\|_{X^{\delta}_{s, b}\times X^{\alpha, \delta}_{k, b}}\lesssim \|(\phi,\psi)\|_{H^{s}\times H^{k}},
		\end{equation*}
		where the norms $\|\cdot\|_{X^{ \delta}_{s, b}}$ and  $\|\cdot\|_{X^{\alpha, \delta}_{s, b}}$ are defined in \eqref{bnorm}.
	\end{theorem}
	\begin{remark}\label{obsalpha}The case $\alpha \in (0,1)$ and $k=s$, with $s>-1/2$, of this theorem, was proved in \protect\cite{CP}. Also, they observed that the LWP in the case $0<\alpha<1$ is equivalent to the LWP in the case $\alpha>1$ by using the transformation $v(x,t):=\tilde{v}(\alpha^{-1/3}x, t)$ and $u(x,t):=\tilde{u}(\alpha^{-1/3}x, t)$ where 
		\begin{equation*}
			\begin{cases}
				\partial_t\tilde{v} + \frac{1}{\alpha}\partial_x^3\tilde{v} + \partial_x(\tilde{v}\tilde{w}^2) =0,&\\
				\partial_t\tilde{w} +\partial_x^3\tilde{w} + \partial_x(\tilde{v}^2\tilde{w}) =0.& 
			\end{cases}
		\end{equation*}
		So we restrict ourselves to prove Theorem \ref{loc-sys} in the case $\alpha\in(-\infty,\,0)\cup(1,\,+\infty)$.\end{remark}

	The main ingredients in the proof of Theorem \ref{loc-sys} are the new trilinear estimates:
	
	\begin{proposition}\label{prop1}
		Let $\alpha \neq 0,1$,  $b=1/2+\epsilon$, and $b'=-1/2+2\epsilon$, with $0<\epsilon<\min\left\{\frac{2s+1}{15},\frac{1}{6}\right\}$. Then  the following trilinear estimates 
		\begin{equation}\label{tlint-m1}
			\|(vw_1 w_2)_x\|_{X_{s,b'}}\lesssim  \|v\|_{X_{s,b}} \|w_1\|_{X^{\alpha}_{k,b}}\|w_2\|_{X^{\alpha}_{k,b}}
		\end{equation}
		and
		\begin{equation}\label{tlint-m2}
			\|(v_1 v_2 w)_x\|_{X^{\alpha}_{k,b'}}\lesssim  \|v_1\|_{X_{s,b}}\|v_2\|_{X_{s,b}} \|w\|_{X^{\alpha}_{k,b}},
		\end{equation}
		holds for any  $s, k$   in the following region: $s,k  >-\frac12$ and $|s-k|\leq 1/2$. Moreover \eqref{tlint-m1} also hold if 
		$s=-1/2$ and $-1/2<k$ and  \eqref{tlint-m2} also hold if $k=-1/2$ and $-1/2<s$.
	\end{proposition}	
	Also, we establish some ill-posedness results. The first one is about the smoothness of the solution mapping  associated to the system \eqref{ivp-sy}.

	\begin{theorem}\label{mainTh-ill1.1}
		Let $\alpha \neq  0,1$. For any $s< -1/2$ or $k<-1/2$ or $|s-k|>2$ and for given  $(\phi, \psi)\in H^s(\R)\times H^k(\R)$, there exist no time $T =T(\|(\phi, \psi)\|_{H^s\times H^k})$ such that the solution mapping that takes initial data $(\phi, \psi)$ to the solution $(v, w) \in C([0,T];H^s) \times C([0,T];H^k)$ to the IVP \eqref{ivp-sy} is $C^3$ at the origin.
	\end{theorem}
	\begin{remark} We point out that in Section \ref{ill} we prove a little bit more stronger result than  the Theorem \ref{mainTh-ill1.1}. 
		
	\end{remark}
	%
	%
In the next figure we represent the regions where we have L.W.P. and $C^3$-ill-posedness:
%

%
%
%
%
\begin{figure}[!htb]
	\centering
	\includegraphics[scale=0.07]{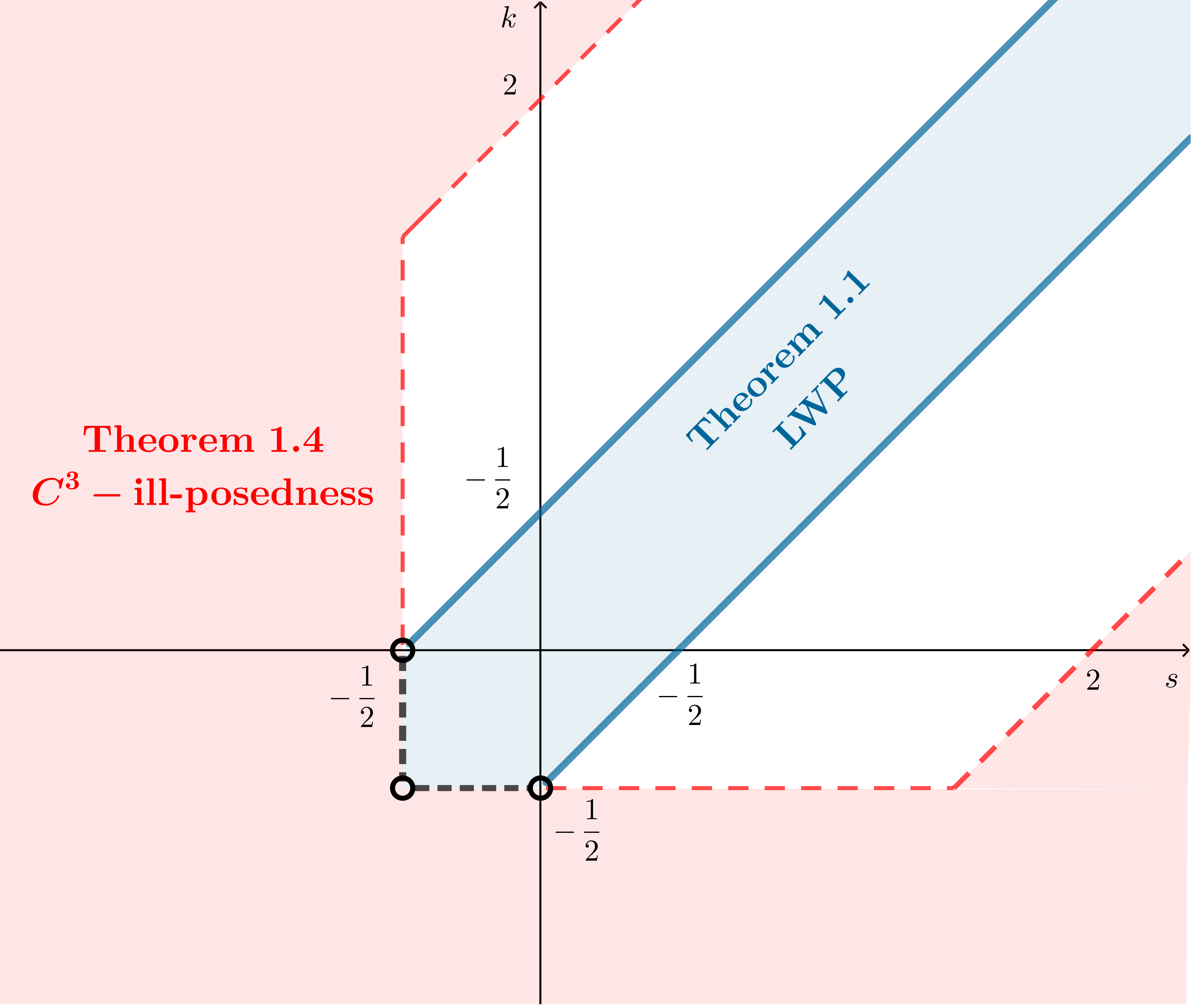}
\end{figure}
%
%
%
%
%


	The second one is about the failure of the trilinear estimates \eqref{tlint-m1} and \eqref{tlint-m2}. We prove the following results:

	\begin{proposition}\label{prop1ill}
		Let  $\alpha \neq 0,1$. \\
		(a) The trilinear estimate \eqref{tlint-m1}  fail to hold for any $b \in \R$ whenever $s-2k>1$  or $k<-1/2$.\\
		(b) The trilinear estimate  \eqref{tlint-m2} fail to hold for any $b \in \R$ whenever $k-2s>1$  or $s<-1/2$.
	\end{proposition}
	and 
	\begin{proposition}\label{prop1illx1}
		Let  $\alpha \neq 0,1$. \\
		(a) The trilinear estimate \eqref{tlint-m1}  fail to hold  whenever $s-k>2$, for any $\epsilon$ such that $0<\epsilon <\frac23(s-k-2)$.\\
		(b) The trilinear estimate \eqref{tlint-m2}  fail to hold  whenever $k-s>2$, for any $\epsilon$ such that $0<\epsilon <\frac23(s-k-2)$.
	\end{proposition}

	Also, at the endpoint we have
	
	\begin{proposition}\label{trilinendpoint}
		Let $\alpha \neq 0,1$, then the estimate
		\begin{equation}
			\|(vw_1 w_1)_x\|_{X_{-\frac{1}2,-\frac12 +4\epsilon}}\lesssim \|v\|_{X_{-\frac{1}2,\frac12 +\epsilon}}\|w_1\|_{X^ \alpha_{-\frac{1}2,\frac12 +\epsilon}}\|w_2\|_{X^ \alpha_{-\frac{1}2,\frac12 +\epsilon}} \label{x2.3 1}
		\end{equation}
		and
		\begin{equation}
			\|(v_1 v_2 w)_x\|_{X^\alpha_{-\frac{1}2,-\frac12 +4\epsilon}}\lesssim \|v_1\|_{X_{-\frac{1}2,\frac12 +\epsilon}}\|v_2\|_{X_{-\frac{1}2,\frac12 +\epsilon}}\|w\|_{X^ \alpha_{-\frac{1}2,\frac12 +\epsilon}} \label{x2.3 2}
		\end{equation}
		fails to hold whenever $\epsilon>0$.
	\end{proposition}
	
	The plan of this paper is as follows. In Section 2 we fix some notations, define the spaces when we perform the
	iteration process, recall some useful inequalities. In Section 3 we prove the crucial result: Proposition \ref{prop1}. In Section 4 we prove the Theorem \ref{loc-sys}. 
	In Section 5 we prove the ill-posedness results and finally, in the Section 6 we give some negative result related to the approach adopted in Section 3.

\secao{Function spaces and preliminary estimates}

In this section we fix some notations, define the function spaces and remember some preliminary results. First, we introduce the integral equations associated to the system \eqref{ivp-sy},
\begin{align}
	v(t)&=U(t)\phi-\int_0^tU(t-t')\partial_x(vw^2)(t')dt'\label{eqint1},\\
	w(t)&=U^{\alpha}(t)\psi -\int_0^tU^{\alpha}(t-t')\partial_x(v^2w)(t')dt'\label{eqint2},
\end{align}
where $U^{\alpha}(t):=e^{-t\alpha\partial_x^3}$ is the unitary group associated to the linear problem $\partial_t u+\alpha\partial_x^3u=0$ and defined via Fourier transform by $U^{\alpha}(t)\phi=\{ e^{it\alpha(\cdot)^3}\hat{\phi}(\cdot) \}\check{\,} $. Here $U(t)$ denotes $U^1(t)$. In order to use the Fourier restriction norm method and prove the local result,  we introduce the Bourgain space  $X^{\alpha}_{s,b}$, for $s,b\in \R$, to be the completion of the Schwartz class $\mathcal{S}(\R^2)$ under the norm
\begin{equation}\label{bnorm}
	\|f\|_{X^{\alpha}_{s,b}}:=\|U^{\alpha}(t)f\|_{H^{b}_t(\R;H^s_x)}=\|\langle \xi\rangle^s\langle \tau-\alpha\xi^3\rangle^{b}
	\tilde{f}(\tau,\xi) \|_{L^2_{\tau,\xi}},
\end{equation}
where $\langle\cdot \rangle:=1+|\cdot|$ and $\tilde{f}$ is the Fourier transform in $(t,x)$ variable
$$\tilde{f}(\tau,\xi):=c\int_{\R^2}e^{-i(x\xi+t\tau)}f(t,x)dtdx.$$
Hereafter,  for $\alpha=1$ we will use $X_{s,b}$ instead of $X^{1}_{s,b}$.
If $b>1/2$, we have that $X_{s,b}^{\alpha}\hookrightarrow C(\R: H^{s}_x(\R))$ and thus for an interval $I=[-\delta,\,\delta]$, we can define the restricted bourgain spaces $X^{\alpha,\delta}_{s,b}$ endowed with the norm $$\|f\|_{X^{\alpha,\delta}_{s,b}}=\inf\{\|g\|_{X^{\alpha}_{s,b}};\, g|_{[-\delta,\,\delta]}=f   \}.$$
Of course, we write $X^{\delta}_{s,b}$ instead of $X^{1,\delta}_{s,b}$.

Now we remember some linear estimates, important parts when we use Fourier Restriction norm method. Let $\eta$ a smooth function supported on the interval $[-2,\,2]$ such that $\eta(t)=1$ for all $t\in [-1,\,1]$. We denote, for each $\delta >0$, $\eta_{\delta}(t)=\eta(t/\delta)$. The following estimates holds (see e.g. \protect\cite{KPV1996} or \protect\cite{GTV1997})
\begin{lemma}\label{lemalin} Let $\delta>0$, $s\in\R$ and $-1/2<b'\leq 0\leq b\leq b' +1$. Then we have
	\begin{enumerate}
		\item[(i)] \ \ \ $\displaystyle\|\eta(t) U^{\alpha}(t)\phi\|_{X_{s,b}^{\alpha}}\lesssim \|\phi\|_{H^s}$.
		\ \\
		
		\item[(ii)] \ \ \ $\left\| \eta_{\delta}(t)\displaystyle\int_0^tU^{\alpha}(t-t')f(t')dt'  \right\|_{X^{\alpha}_{s,b}}\lesssim \delta^{1-b+b'}\|f\|_{X^{\alpha}_{s,b'}}$.
	\end{enumerate}
	
\end{lemma}
The following lemma will be useful in the proof of the trilinear estimates
\begin{lemma}\label{lemxm1}
	\begin{itemize}
		\item[(i)] If $a,b>0$ and $a+b>1$, we have
		\begin{equation}\label{lfc1-1}
			\int_{\R} \dfrac{dx}{\langle x -\alpha \rangle^{a}\langle x -\beta \rangle^{b}} \lesssim \dfrac{1}{\langle \alpha -\beta \rangle^{c}},  \quad c=\min\{a,b, a+b-1\}.
		\end{equation}
		
		\item[(ii)]
		Let $a, \eta \in \R$, $a, \eta  \neq 0$, $b>1$, then
		\begin{equation}\label{xlfc1-2}
			\int_{\R} \dfrac{dx}{\langle a(x^2-\eta^2) \rangle^{b}}  \lesssim \dfrac{1}{|a\eta|}.
		\end{equation}
		
		\item[(iii)]
		Let $a, \eta \in \R$, $a, \eta  \neq 0$, $b>1$, then
		\begin{equation}\label{x1lfc1-2}
			\int_{\R} \dfrac{|x\pm \eta|\,dx}{\langle a(x^2\pm\eta^2) \rangle^{b}} \lesssim \dfrac{1}{|a|}.
		\end{equation}

		\item[(iv)] For $l>1/3$,
		\begin{equation}\label{x2lfc1-2}
			\int_{\R} \dfrac{dx}{\langle x^3+a_2x^2 +a_1 x+a_0 \rangle^{l}} \lesssim 1.
		\end{equation}
		
	\end{itemize}
\end{lemma}
\begin{proof}
	The proof of \eqref{lfc1-1}  can be found in \protect\cite{Takaoka2000},  \eqref{xlfc1-2} and \eqref{x1lfc1-2} in \protect\cite{Carvajal2004} and \eqref{x2lfc1-2}  in  \protect\cite{BOP1997}.
\end{proof}
%
%

%
%

\secao{Trilinear estimates: proof of Proposition \ref{prop1}}\label{tri}

In this section the ideas in \protect\cite{Tao2001} plays a central role in the proof of Proposition \ref{prop1}. We remember some notations, results and follows the arguments contained therein.  Let $k\geq 2$ be an integer, a $[k;\R^{d+1}]$-multiplier is any function $m:\Gamma_k(\R^{d+1})\to\mathbb{C}$, where $\Gamma_k(\R^{d+1})$ denotes  the hyperplane $\Gamma_k(\R^{d+1})=\{(\xi_1,\ldots,\xi_k)\in (\R^{d+1})^k;\, \xi_1+\cdots+\xi_n=0\}$ endowed with the measure
$$\int_{_{\Gamma_k(\R^{d+1})}}f:=\int_{(\R^d)^{k-1}}f(\xi_1,\ldots,\xi_{k-1},-\xi_1-\cdots-\xi_k-1)d\xi_1d\xi_2\cdots d\xi_{k-1}.$$
The norm of a $[k;\R^{d+1}]$-multiplier $m$, denoted by $\|m\|_{[k;\R^{d+1}]}$, is the best constant such that 
$$\left|  \int_{\Gamma_k(\R^{d+1})}m(\xi)\prod_{j=1}^{n}f_j(\xi_j) \right|\leq \|m\|_{[k;\R^{d+1}]}\prod_{j=1}^{k}\|f_j\|_{L^2(\R^{d+1})},$$
holds for all test functions $f_j$ on $\R^{d+1}$.

Now we start to work on our trilinear estimates. By duality and Plancherel (see e.g. \protect\cite{Tao2001} or \protect\cite{CP2019}),  one can see that  the estimate \eqref{tlint-m1} is equivalent to
\begin{equation}\label{tln-e6}
	\Big|\int_{\xi_1+\cdots+\xi_4=0\atop \tau_1+\cdots+\tau_4=0}
	m(\xi_1,\tau_1,\cdots,\xi_4,\tau_4)\Pi_{j=1}^4\widetilde{f}_j(\xi_j, \tau_j)\Big| \lesssim \Pi_{j=1}^4\|f_j\|_{L^2_{\xi\tau}},
\end{equation}
where
\begin{equation}\label{tln-e7}
	m(\xi_1,\tau_1,\cdots,\xi_4,\tau_4):=\frac{\xi_4\,\langle\xi_4\rangle^{s}}{\langle\xi_1\rangle^s\langle\xi_2\rangle^k\langle\xi_3\rangle^k\langle\tau_1-\xi_1^3\rangle^{\frac12+\epsilon}\langle\tau_2-\alpha\xi_2^3\rangle^{\frac12+\epsilon}\langle\tau_3-\alpha\xi_3^3\rangle^{\frac12+\epsilon}\langle\tau_4-\xi_4^3\rangle^{\frac12-2\epsilon}}.
\end{equation}

In this way, recalling the definition of the norm $\|m\|_{[4;\R^2]}$ of the multiplier $m$,  the whole matter reduces to showing that
\begin{equation}\label{tln-e8}
	\|m\|_{[4;\R^2]}\lesssim 1.
\end{equation}
Observe that
\begin{equation}\label{xi-4}
	\xi_4\,\langle\xi_4\rangle^{s} \leq  \langle\xi_4\rangle^{s+1}\leq \langle\xi_4\rangle^{1/2}\langle\xi_4\rangle^{s+1/2}\leq \langle\xi_4\rangle^{1/2}( \langle\xi_1\rangle^{s+1/2}+\langle\xi_2\rangle^{s+1/2}+\langle\xi_3\rangle^{s+1/2}).
\end{equation}
We define
\begin{equation}\label{m1}
	m_1(\xi_1, \tau_1, \xi_2,\tau_2)=\frac{\langle\xi_1\rangle^{\frac12}}{\langle\tau_1-\xi_1^3 \rangle^{\frac12-2\epsilon} \langle\xi_2\rangle^k \langle\tau_2-\alpha\xi_2^3\rangle^{\frac12+\epsilon}},
\end{equation}
\begin{equation}\label{m12}
	m_2(\xi_1, \tau_1, \xi_2,\tau_2)=\frac{\langle\xi_1\rangle^{s-k+\frac12}}{\langle\tau_1-\alpha\xi_1^3 \rangle^{\frac12+\epsilon} \langle\xi_2\rangle^s \langle\tau_2-\xi_2^3\rangle^{\frac12+\epsilon}}.
\end{equation}
From \eqref{xi-4} and \eqref{tln-e7}, we get
\begin{equation}\label{mult-1}
	\begin{split}
		m\leq &\frac{\langle\xi_4\rangle^{\frac12}}{\langle\tau_4-\xi_4^3\rangle^{\frac12-2\epsilon}\langle\xi_3\rangle^k\langle\tau_3-\alpha\xi_3^3\rangle^{\frac12+\epsilon}}\frac{\langle\xi_1\rangle^{\frac12}}{\langle\tau_1-\xi_1^3 \rangle^{\frac12+\epsilon} \langle\xi_2\rangle^k \langle\tau_2-\alpha\xi_2^3\rangle^{\frac12+\epsilon}} \\
		&+ \frac{\langle\xi_4\rangle^{\frac12}}{\langle\tau_4-\xi_4^3\rangle^{\frac12-2\epsilon}\langle\xi_3\rangle^k\langle\tau_3-\alpha\xi_3^3\rangle^{\frac12+\epsilon}} \frac{\langle\xi_2\rangle^{s-k+\frac12}}{\langle\tau_2-\alpha\xi_2^3\rangle^{\frac12+\epsilon}\langle\xi_1\rangle^s\langle\tau_1-\xi_1^3\rangle^{\frac12+\epsilon}
		}\\
		&+ \frac{\langle\xi_4\rangle^{\frac12}}{\langle\tau_4-\xi_4^3\rangle^{\frac12-2\epsilon}\langle\xi_2\rangle^k\langle\tau_2-\alpha\xi_2^3\rangle^{\frac12+\epsilon}}\frac{\langle\xi_3\rangle^{s-k+\frac12}}{\langle\tau_3-\alpha\xi_3^3\rangle^{\frac12+\epsilon}\langle\xi_1\rangle^s\langle\tau_1-\xi_1^3\rangle^{\frac12+\epsilon}}\\
		=:&J_1+J_2+J_3.
	\end{split}
\end{equation}
Therefore, we have
\begin{equation}\label{J1}
	\begin{split}
		J_1
		&\leq m_1(\xi_4,\tau_4,\xi_3,\tau_3)\, m_1(\xi_1, \tau_1, \xi_2,\tau_2)
	\end{split}
\end{equation}
\begin{equation}\label{J2}
	\begin{split}
		J_2
		&\leq m_1(\xi_4,\tau_4,\xi_3,\tau_3)\,m_2(\xi_2,\tau_2,\xi_1,\tau_1),
	\end{split}
\end{equation}
\begin{equation}\label{J3}
	\begin{split}
		J_3&\leq m_1(\xi_4,\tau_4,\xi_2,\tau_2)\,m_2(\xi_3,\tau_3,\xi_1,\tau_1).
	\end{split}
\end{equation}

Now, using comparison principle, permutation and composition properties (see respectively Lemmas 3.1, 3.3 and 3.7 in \cite{Tao2001}), it is enough to bound $\|m_j\|_{[3;\R^2]}$, $j=1,2$, or equivalently,  
to show the following bilinear estimates
\begin{equation}\label{bil-1}
	\|uv\|_{L^2(\R^2)}\lesssim \|u\|_{X_{-\frac12, \frac12-2\epsilon}}\|v\|_{X^{\alpha}_{k, \frac12+\epsilon}},
\end{equation}
and
\begin{equation}\label{bil-2}
	\|uv\|_{L^2(\R^2)}\lesssim \|u\|_{X^{\alpha}_{k-s-\frac12, \frac12+\epsilon}}\|v\|_{X_{s, \frac12+\epsilon}}.
\end{equation}
This equivalence can be proved using again duality and a similar calculations as the ones used to obtain \eqref{tln-e6}.


Similarly,  the estimate \eqref{tlint-m2} is equivalent to
\begin{equation}\label{tln-e6x}
	\Big|\int_{\xi_1+\cdots+\xi_4=0\atop \tau_1+\cdots+\tau_4=0}
	M(\xi_1,\tau_1,\cdots,\xi_4,\tau_4)\Pi_{j=1}^4\widetilde{f}_j(\xi_j, \tau_j)\Big| \lesssim \Pi_{j=1}^4\|f_j\|_{L^2_{\xi\tau}},
\end{equation}
where
\begin{equation}\label{tln-e7x}
	M(\xi_1,\tau_1,\cdots,\xi_4,\tau_4):=\frac{\xi_4\,\langle\xi_4\rangle^{k}}{\langle\xi_1\rangle^s\langle\xi_2\rangle^s\langle\xi_3\rangle^k\langle\tau_1-\xi_1^3\rangle^{\frac12+\epsilon}\langle\tau_2-\xi_2^3\rangle^{\frac12+\epsilon}\langle\tau_3-\alpha\xi_3^3\rangle^{\frac12+\epsilon}\langle\tau_4-\alpha \xi_4^3\rangle^{\frac12-2\epsilon}}.
\end{equation}

In this way, recalling the definition of the norm $\|M\|_{[4;\R^2]}$ of the multiplier $M$,  the whole matter reduces to showing that
\begin{equation}\label{tln-e8x}
	\|M\|_{[4;\R^2]}\lesssim 1.
\end{equation}
Observe that
\begin{equation}\label{xi-5}
	\xi_4\,\langle\xi_4\rangle^{k} \leq  \langle\xi_4\rangle^{k+1}\leq \langle\xi_4\rangle^{1/2}\langle\xi_4\rangle^{k+1/2}\leq \langle\xi_4\rangle^{1/2}( \langle\xi_1\rangle^{k+1/2}+\langle\xi_2\rangle^{k+1/2}+\langle\xi_3\rangle^{k+1/2}).
\end{equation}
We define
\begin{equation}\label{m1x}
	M_1(\xi_1, \tau_1, \xi_2,\tau_2)=\frac{\langle\xi_1\rangle^{\frac12}}{\langle\tau_1-\alpha\xi_1^3 \rangle^{\frac12-2\epsilon} \langle\xi_2\rangle^s \langle\tau_2-\xi_2^3\rangle^{\frac12+\epsilon}},
\end{equation}
\begin{equation}\label{m12x}
	M_2(\xi_1, \tau_1, \xi_2,\tau_2)=\frac{\langle\xi_1\rangle^{k-s+\frac12}}{\langle\tau_1-\xi_1^3 \rangle^{\frac12+\epsilon} \langle\xi_2\rangle^k \langle\tau_2-\alpha\xi_2^3\rangle^{\frac12+\epsilon}}.
\end{equation}
From \eqref{xi-5} and \eqref{tln-e7x}, we get
\begin{equation}\label{mult-1a}
	\begin{split}
		M\leq &\frac{\langle\xi_4\rangle^{\frac12}}{\langle\tau_4-\alpha\xi_4^3\rangle^{\frac12-2\epsilon}\langle\xi_2\rangle^s\langle\tau_2-\xi_2^3\rangle^{\frac12+\epsilon}}\frac{\langle\xi_1\rangle^{k-s+\frac12}}{\langle\tau_1-\xi_1^3 \rangle^{\frac12+\epsilon} \langle\xi_3\rangle^k \langle\tau_3-\alpha\xi_3^3\rangle^{\frac12+\epsilon}} \\
		&+ \frac{\langle\xi_4\rangle^{\frac12}}{\langle\tau_4-\alpha\xi_4^3\rangle^{\frac12-2\epsilon}\langle\xi_1\rangle^s\langle\tau_1-\xi_1^3\rangle^{\frac12+\epsilon}} \frac{\langle\xi_2\rangle^{k-s+\frac12}}{\langle\tau_2-\xi_2^3\rangle^{\frac12+\epsilon}\langle\xi_3\rangle^k\langle\tau_3-\alpha \xi_3^3\rangle^{\frac12+\epsilon}
		}\\
		&+ \frac{\langle\xi_4\rangle^{\frac12}}{\langle\tau_4-\alpha\xi_4^3\rangle^{\frac12-2\epsilon}\langle\xi_1\rangle^s\langle\tau_1-\xi_1^3\rangle^{\frac12+\epsilon}}\frac{\langle\xi_3\rangle^{\frac12}}{\langle\tau_3-\alpha\xi_3^3\rangle^{\frac12+\epsilon}\langle\xi_2\rangle^s\langle\tau_2-\xi_2^3\rangle^{\frac12+\epsilon}}\\
		=:&I_1+I_2+I_3.
	\end{split}
\end{equation}
Therefore, we have
\begin{equation}\label{J1x}
	\begin{split}
		I_1
		&\leq M_1(\xi_4, \tau_4, \xi_2,\tau_2)\,M_2(\xi_1,\tau_1,\xi_3,\tau_3),
	\end{split}
\end{equation}
\begin{equation}\label{J2x}
	\begin{split}
		I_2
		&\leq M_1(\xi_4,\tau_4,\xi_1,\tau_1)\,M_2(\xi_2,\tau_2,\xi_3,\tau_3),
	\end{split}
\end{equation}
\begin{equation}\label{J3x}
	\begin{split}
		I_3&\leq M_1(\xi_4,\tau_4,\xi_1,\tau_1)\,M_1(\xi_3,\tau_3,\xi_2,\tau_2).
	\end{split}
\end{equation}
Analogously,  to prove that $\|M_j\|_{[3;\R^2]}\lesssim 1$, $j=1,2$ is equivalent respectively to show the following bilinear estimates
\begin{equation}\label{bil-1x}
	\|uv\|_{L^2(\R^2)}\lesssim \|u\|_{X^{\alpha}_{-\frac12, \frac12-2\epsilon}}\|v\|_{X_{s, \frac12+\epsilon}},
\end{equation}
and
\begin{equation}\label{bil-2x}
	\|uv\|_{L^2(\R^2)}\lesssim \|u\|_{X_{s-k-\frac12, \frac12+\epsilon}}\|v\|_{X^{\alpha}_{k, \frac12+\epsilon}}.
\end{equation}
%

Therefore, the proof of Proposition \ref{prop1} follows if we proof  the four bilinear estimates \eqref{bil-1}, \eqref{bil-2}, \eqref{bil-1x}  and \eqref{bil-2x}. In fact we prove the following propositions
\begin{proposition}\label{ProposCP}
	Let $s>-1/2$,  $\alpha\neq0,1$ and $0<\epsilon<\min\left\{\frac{2s+1}{15},\frac{1}{6}\right\}$. Then we have the bilinear estimates 
	
	\begin{equation}\label{bil-1.1}
		\|fg\|_{L^2(\R^2)}\lesssim \|f\|_{X_{-\frac12, \frac12-2\epsilon}}\|g\|_{X^{\alpha}_{s, \frac12+\epsilon}},
	\end{equation}
	and
	\begin{equation}\label{bil-2.1}
		\|fg\|_{L^2(\R^2)}\lesssim \|f\|_{X^{\alpha}_{-\frac12, \frac12-2\epsilon}}\|g\|_{X_{s, \frac12+\epsilon}}.
	\end{equation} 
\end{proposition}

\begin{proposition}\label{bil-LRX-1} Let $\alpha\in (-\infty,\,0)\cup(1,\,+\infty)$ and $0<\epsilon<\min\left\{\frac{2s+1}{15},\frac{1}{6}\right\}$. Then\\
	\noindent(a)  The inequality \eqref{bil-2} holds for any $(s,k)$ in the region:
	$$ R_1= \left\{(s,k);\,\, k,s>-1/2, \,\, s-k\leq 1/2\,\right\} \cup \left\{(s,k); \,\, -1/2< k, \,\, s=-1/2\,\right\}.$$
	\noindent(b)  The inequality \eqref{bil-2x}	holds for any  $(s,k)$ in the region:   
	$$  R_2=\left\{(s,k); \,\, k,s>-1/2, \,\, s-k\geq -1/2\,\right\} \cup \left\{(s,k); \,\, -1/2< s, \,\, k=-1/2\,\right\}.$$
	
\end{proposition}

Before we prove these results, we establish some preliminary results.

\begin{lemma}\label{lemma1.5}Let $\alpha<1$ ($\alpha\neq0$), $s>-1/2$ and $0<\epsilon<\min\left\{\frac{2s+1}{15},\frac{1}{6}\right\}$. Then we have
	\begin{equation}\label{X1}
		\sup  \limits_{\xi,\tau} 
		\int \limits_{\R} \frac{\langle \xi_2\rangle}{\langle\xi_1 \rangle^{2s}\langle \tau-\xi_2^ 3 -\alpha\xi_1^ 3\rangle^{1-4\epsilon}} d\xi_1\lesssim 1,
	\end{equation}
	where $\xi_2:=\xi-\xi_1$, for a fixed $\xi\in\R$.
\end{lemma}
\begin{proof}
	For the case $0<\alpha <1$, see  in \cite{CP2019}, Lemma 3.2, the estimative of the item labeled by them as (3.12). We will prove the case $\alpha< 0$. We denote by $L_1$ the integral in \eqref{X1}. For a fixed  $\xi$ and $\tau$, let
	\begin{equation}
		H(\xi_1):=\tau-\xi_2^3-\alpha \xi_1^ 3=\tau-\xi_2^3+|\alpha|\xi_1^ 3.
	\end{equation}
	We have 
	$$H'(\xi_1)=3\left[ \xi^ 2_2+|\alpha| \xi_1^ 2\right]> 0.$$
	Thus, the function $\xi_1\mapsto H(\xi_1)$ is monotone on $\mathbb R$. We divide the proof into the following two cases:

	\noindent\textbf{Case 1.} $\bm{(|\xi|<2|\xi_1|)}$ In this case, we have $\langle \xi _2\rangle\lesssim\langle \xi_1\rangle$.  Thus
	\begin{equation}
		\begin{array}{rcl}
			\chi_{\{|\xi|<2|\xi_1|\}}L_1&=&\displaystyle \int \limits_{|\xi_1|> |\xi|/2} \frac{\langle \xi_2\rangle \langle \xi_1 \rangle}{\langle\xi_1 \rangle^{2s+1}\langle \tau-\xi_2^ 3 +|\alpha|\xi_1^ 3\rangle^{1-4\epsilon}} d\xi_1\lesssim\displaystyle  \int \limits_{\R} \frac{\langle \xi_1 \rangle^ 2}{\langle\xi_1 \rangle^{2s+1}\langle H(\xi_1)\rangle^{1-4\epsilon}} d\xi_1\\
			&=&\displaystyle  \int \limits_{\mathbb R} \frac{1}{\langle\xi_1 \rangle^{2s+1}\langle H(\xi_1)\rangle^{1-4\epsilon}} d\xi_1+ \displaystyle  \int \limits_{\R} \frac{\xi_1^ 2}{\langle\xi_1 \rangle^{2s+1}\langle H(\xi_1)\rangle^{1-4\epsilon}} d\xi_1\\
			&=&J_1+J_2.
		\end{array}
	\end{equation}
	Using \eqref{x2lfc1-2} we have $J_1 \lesssim 1$ provided $0<\epsilon< \frac{1}{6}$. In what follows, we estimate $J_2$ integrating over: \textbf{(a)} $|\xi_1|<1 $ and \textbf{(b)} $|\xi_1|\ge 1$, separetely. 
	In the first situation, using that $2s+1\geq 0$, we have 
	\begin{equation}
		\begin{array}{rcl}
			\chi_{\{|\xi_1|<1\}}J_2&\lesssim& \displaystyle \int \limits_{|\xi_1|<1} \frac{1}{\langle H(\xi_1) \rangle^{1-4\epsilon} }d\xi_1\lesssim 1.
		\end{array}
	\end{equation}
	For the second case, considering the sets $$A=\left\{\xi_1: \langle H(\xi_1) \rangle \lesssim |\xi_1|^3\right\}\ \ \ \textrm{and}\ \ \  B=\left\{\xi_1: \langle H(\xi_1) \rangle \gtrsim |\xi_1|^3\right\},$$ we have 
	\begin{equation}\label{3.18}
		\begin{array}{rcl}
			\chi_{\{|\xi_1| \geq 1\}}J_2&=& \displaystyle  \int \limits_{|\xi_1| \geq 1} \frac{\xi_1^ 2}{\langle\xi_1 \rangle^{2s+1}\langle H(\xi_1)\rangle^{1-4\epsilon}}\chi_{_{A}}(\xi_1) d\xi_1+\displaystyle  \int \limits_{|\xi_1| \geq 1} \frac{\xi_1^ 2}{\langle\xi_1 \rangle^{2s+1}\langle H(\xi_1)\rangle^{1-4\epsilon}}\chi_{_{B}}(\xi_1) d\xi_1\\
			&\lesssim & \displaystyle \int \limits_{|\xi_1|\geq 1} \frac{H'(\xi_1) \langle H(\xi_1)\rangle^{5\epsilon}}{\langle\xi_1 \rangle^{2s+1}\langle H(\xi_1)\rangle^{1+\epsilon}}\chi_{_{A}}(\xi_1) d\xi_1+\displaystyle \int \limits_{|\xi_1|\geq 1} \frac{|\xi_1|^ 2}{|\xi_1|^{2s+1}|\xi_1|^{3-12\epsilon}} d\xi_1\\
			&\lesssim & \displaystyle \int \limits_{\R} \frac{H'(\xi_1) |\xi_1|^{15\epsilon}}{\langle\xi_1 \rangle^{2s+1}\langle H(\xi_1)\rangle^{1+\epsilon}} d\xi_1+\displaystyle \int \limits_{|\xi_1|\geq 1} \frac{|\xi_1|^ 2}{|\xi_1|^{2s+1}|\xi_1|^{3-12\epsilon}} d\xi_1,
		\end{array}
	\end{equation}
	where in the first integral we use that $\xi_1^ 2\lesssim H'(\xi_1)$. Obviously, the second integral is $\lesssim 1$, provided $0<\epsilon<\frac{2s+1}{12}$. For the first integral, performing the change of variables $x=H(\xi_1)$ on $\R$, if $0<\epsilon<\frac{2s+1}{15}$, we obtain  
	\begin{equation}
		\displaystyle \int \limits_{\R} \frac{H'(\xi_1) |\xi_1|^{15\epsilon}}{\langle\xi_1 \rangle^{2s+1}\langle H(\xi_1)\rangle^{1+\epsilon}} d\xi_1		\lesssim \int \limits_{\R} \frac{dx}{\langle x \rangle^{1+\epsilon}}\lesssim 1.
	\end{equation}
	\noindent\textbf{Case 2.} $\bm{(2|\xi_1|\leq |\xi|)}$ Because $\xi_1+\xi_2=\xi$, then we have 
	\begin{align}
		\max\left\{\langle \xi_1\rangle,\,\langle \xi_2\rangle\right\}&\lesssim \langle \xi\rangle\label{l01},\\
		\langle \xi_2\rangle \langle \xi \rangle \lesssim  \langle \xi_2 \rangle^ 2&\lesssim 1+H'(\xi_1)\label{l02}.
	\end{align}
	Fix $0<\epsilon<1/6$, $b=(1-4\epsilon)^{-1}$ and $a$ such that  $a+b=3$. Considering the sets  $$A=\left\{\xi_1: \langle H(\xi_1) \rangle \lesssim \langle\xi_1\rangle^a\langle \xi\rangle^b\right\}\ \ \  \textrm{and} \ \ \ B=\left\{\xi_1: \langle H(\xi_1) \rangle \gtrsim \langle\xi_1\rangle^a\langle \xi\rangle^b\right\},$$ we have 
	\begin{align*}
		\chi_{\{2|\xi_1| \leq |\xi|\}} L_1&=\displaystyle \int \limits_{|\xi_1| \leq |\xi| /2} \frac{\langle \xi_2\rangle \langle \xi \rangle}{\langle\xi_1 \rangle^{2s}\langle \xi \rangle \langle H(\xi_1)\rangle^{1-4\epsilon}}\chi_{_{A}}(\xi_1) d\xi_1+\displaystyle \int \limits_{\R} \frac{\langle \xi_2\rangle }{\langle\xi_1 \rangle^{2s} \langle H(\xi_1)\rangle^{1-4\epsilon}}\chi_{_{B}}(\xi_1) d\xi_1\\
		&\lesssim \displaystyle  \int \limits_{\R} \frac{1}{\langle\xi_1 \rangle^{2s+1}\langle H(\xi_1)\rangle^{1-4\epsilon}} \chi_{_{A}}(\xi_1)d\xi_1+ \displaystyle  \int \limits_{|\xi_1| \leq |\xi| /2} \frac{H'(\xi_1)}{\langle\xi_1 \rangle^{2s}\langle \xi \rangle\langle H(\xi_1)\rangle^{1-4\epsilon}} \chi_{_{A}}(\xi_1)d\xi_1\\
		&\ \ \ \ \ + \int_{\R} \dfrac{\langle  \xi\rangle   }{\langle  \xi_1\rangle^{2s} \langle \xi_1\rangle^{a(1-4 \epsilon)} \langle \xi \rangle^{b(1-4 \epsilon)}  } \chi_{_{B}}(\xi_1)d\xi_1\\ 
		&=J_1+J_2+J_3.
	\end{align*}
	By \eqref{x2lfc1-2} we have $J_1 \lesssim 1$ provided $0<\epsilon< \frac{1}{6}$ and $2s+1\ge 0$. For $J_3$, remembering the definitions of $a$ and $b$, if $0<\epsilon<(2s+1)/12$, we have 
	$$J_3\lesssim \int_{\R}\frac{d\xi_1}{\langle \xi_1\rangle^{2+2s-12\epsilon} }\lesssim 1.$$
	For $J_2$, taking account \eqref{l01},  if $0<\epsilon<5/9$, we have $1-5b\epsilon>0$, and so $\langle \xi_1\rangle^{1-5b\epsilon}\lesssim\langle \xi\rangle^{1-5b\epsilon}$. Thus, considering that $2s+1\geq 0$ and the definition of $a$ and $b$, if $0<\epsilon<(2s+1)/15$ we have
	\begin{align*}
		J_2&=	\int_{|\xi_1| \leq |\xi| /2} \dfrac{  H'(\xi_1)  \chi_{_{A}}(\xi_1)}{ \langle \xi_1 \rangle^{2s}\langle \xi \rangle\langle H(\xi_1) \rangle^{1-4\epsilon}  } d\xi_1\lesssim   \int_{|\xi_1| \leq |\xi| /2} \dfrac{ H'(\xi_1) \langle H(\xi_1) \rangle^{5\epsilon} \chi_{_{A}}(\xi_1)}{\langle \xi_1 \rangle^{2s} \langle \xi \rangle \langle H(\xi_1) \rangle^{1+\epsilon}  } d\xi_1\\
		\lesssim &  \int_{|\xi_1| \leq |\xi| /2} \dfrac{ H'(\xi_1)  }{\langle \xi_1 \rangle^{2s-5 a\epsilon} \langle \xi \rangle^{1-5b \epsilon}\langle H(\xi_1) \rangle^{1+\epsilon}  } d\xi_1\lesssim   \int_{\R} \dfrac{ H'(\xi_1)  }{\langle  \xi_1 \rangle^{2s+1-5 (a+b)\epsilon}\langle H(\xi_1) \rangle^{1+\epsilon}  } d\xi_1\\
		\lesssim& \int_{\R}\frac{H'(\xi_1)}{\langle H(\xi_1)\rangle^{1+\epsilon}}d\xi_1.
	\end{align*}
	Again, making the  change  of variables  $x=H (\xi_1)$, we get the desired bound.
\end{proof}
Now we are in position to prove Proposition \ref{ProposCP}.
\begin{proof}[Proof of the Proposition \ref{ProposCP}] First of all, we remember that the case $0<\alpha<1$ was proved in \cite{CP2019}. So, let's assume that $\alpha\in(-\infty,0)\cup(1,\infty)$.
	We start to prove  that the inequality \eqref{bil-1.1}  hold  if $\alpha>1$. Let  $u \in X_{-\frac12, \frac12-2\epsilon}$ and $v\in X^{\alpha}_{s, \frac12+\epsilon}$ with $\epsilon>0$ and $s>- \frac12$. Considering $f$ and $g$ such that
	\begin{equation*}
		u(x,t)=f(\alpha^{-1/3}x,t), \quad v(x,t)=g(\alpha^{-1/3}x,t),
	\end{equation*}
	using that $\langle a \xi \rangle \sim_{a} \langle \xi \rangle$, for $a\neq 0$, and scaling properties of the Fourier transform, we have
	$$
	\|f\|_{X_{-\frac12, \frac12-2\epsilon}^{1/\alpha}}\sim_{\alpha} \|u\|_{X_{-\frac12, \frac12-2\epsilon}}\quad \textrm{and}\quad \|g\|_{X_{s, \frac12+\epsilon}}\sim_{\alpha} \|v\|_{X^{\alpha}_{s, \frac12+\epsilon}},
	$$
	thus $f \in X_{-\frac12, \frac12-2\epsilon}^{1/\alpha}$ and $g\in X_{s, \frac12+\epsilon}$. Because $1/\alpha \in (0,1)$ we can apply the estimate \eqref{bil-2.1}  to obtain
	\begin{equation}\label{bil-3.1}
		\begin{split}
			\|fg\|_{L^2(\R^2)}&\lesssim \|f\|_{X^{1/\alpha}_{-\frac12, \frac12-2\epsilon}}\|g\|_{X_{s, \frac12+\epsilon}}\\
			&\lesssim_{\alpha} \|u\|_{X_{-\frac12, \frac12-2\epsilon}} \|v\|_{X^{\alpha}_{s, \frac12+\epsilon}}.
		\end{split}
	\end{equation}
	We concludes that \eqref{bil-1.1} holds for $\alpha>1$,  observing that $ \|fg\|_{L^2(\R^2)} =\alpha^{-2/3}\|u v\|_{L^2(\R^2)} $. 
	
	Now, if $\alpha<0$, using Plancherel's identity, one can see that the estimate \eqref{bil-1.1} is equivalent to 
	\begin{equation}\label{x3.13}
		\|B_s(f,g)\|_{L^2_\xi L^2_\tau}\leq C\|f\|_{L^2}\|g\|_{L^2},
	\end{equation}
	where
	\begin{equation}
		B_s(f,g)=\int \limits_{\R^2}\frac{\langle \xi_2 \rangle^{\frac{1}{2}} \tilde{f}(\xi_2,\tau_2) \tilde{g}(\xi_1,\tau_1)}{\langle\xi_1 \rangle^{s}\langle\tau_1-\alpha\xi_1^ 3\rangle ^{\frac{1}{2}+\epsilon}\langle \tau_2-\xi_2^ 3 \rangle^{\frac{1}{2}-2\epsilon} }d\xi_1 d\tau_1,\label{3.13}
	\end{equation}
	with $\xi_2=\xi-\xi_1$, $\tau_2=\tau-\tau_1.$
	Using Cauchy-Schwarz inequality we note \eqref{x3.13} holds if 
	\begin{equation}
		\mathcal{L}_1:=\sup\limits_{\xi,\tau}\int \limits_{\R^ 2} \frac{\langle \xi_2\rangle}{\langle\xi_1 \rangle^{2s}\langle\tau_1-\alpha\xi_1^ 3\rangle ^{1+2\epsilon}\langle \tau_2-\xi_2^ 3 \rangle^{1-4\epsilon}} d\xi_1 d\tau_1\lesssim 1.\label{3.14}
	\end{equation}
	In order to see that the estimate \eqref{3.14} hold,  applying the estimate \eqref{lfc1-1} (Lemma \ref{lemxm1}) for  the integral in $\tau_1$, we obtain 
	\begin{equation}
		\mathcal{L}_1\lesssim \sup  \limits_{\xi,\tau} 
		\int \limits_{\R} \frac{\langle \xi_2\rangle}{\langle\xi_1 \rangle^{2s}\langle \tau-\xi_2^ 3 -\alpha\xi_1^ 3\rangle^{1-4\epsilon}} d\xi_1,
	\end{equation}
	if $0<\epsilon<\frac{1}{4}$. Applying now \eqref{X1}, we get the desired bound and finish this case. In the same way we can prove the inequality \eqref{bil-2.1} for $\alpha<0$ or $\alpha>1$.
\end{proof}
The next results are usefull in the proof of the Proposition \ref{bil-LRX-1}.

\begin{lemma}\label{lema02}  Let $l\ge-1/2$ and $b>1/2$.  Considering $F$ a monotone function defined on a Lebesgue-measurable set $X\subset \R$ such that $$ |F'(\xi_1)|\gtrsim \max\{\xi_1^2,\,\xi_2^2\},\ \ \ \forall\xi_1\in X,$$ where $\xi_2=\xi-\xi_1$, for a fixed $\xi\in\R$. Then we have
	\begin{equation}\label{eqlema02}
		\int_{X}\frac{\langle \xi_i\rangle}{\langle \xi_j\rangle^{2l}\langle F(\xi_1)\rangle^{2b}}d\xi_1\lesssim \int_{X}\frac{1}{\langle F(\xi_1)\rangle^{2b}}d\xi_1+1,
	\end{equation}	 
	for all $i,\,j\in\{1,\,2\}$.
\end{lemma}
\begin{proof} Considering $F$ a monotone increasing function, we have that $F'(\xi_1)\gtrsim\xi_k^2$, for $k\in\{1,\,2\}$. So 
	\begin{align*}
		\langle \xi_i\rangle\langle \xi_j\rangle&\lesssim \langle \xi_i\rangle^2+\langle \xi_j\rangle^2\lesssim 1+F'(\xi_1),
	\end{align*}
	for all $i,\,j\in\{1,\,2\}$. Hence
	\begin{align*}
		\int_{X}\frac{\langle \xi_i\rangle}{\langle \xi_j\rangle^{2l}\langle F(\xi_1)\rangle^{2b}}d\xi_1&=\int_{X}\frac{\langle \xi_i\rangle\langle \xi_j\rangle}{\langle \xi_j\rangle^{2l+1}\langle F(\xi_1)\rangle^{2b}}d\xi_1\\
		&\lesssim \int\frac{1}{\langle \xi_j\rangle^{2l+1}\langle F(\xi_1)\rangle^{2b}}d\xi_1+\int_{X}\frac{F'(\xi_1)}{\langle \xi_j\rangle^{2l+1}\langle F(\xi_1)\rangle^{2b}}d\xi_1\\
		&\lesssim \int_{X}\frac{1}{\langle F(\xi_1)\rangle^{2b}}d\xi_1+\int_{X}\frac{F'(\xi_1)}{\langle F(\xi_1)\rangle^{2b}}d\xi_1
	\end{align*}
	and making the change of variable $x=F(\xi_1)$, we finish the proof.
\end{proof}
\begin{lemma}\label{lemma1.7}
	Let $b>1/2$, $l \geq -1/2$ and $\alpha\in (-\infty,\,0)\cup(1,\,+\infty)$, then
	\begin{equation}\label{est5}
		J:=\int_{B}\frac{\langle \xi_2\rangle}{\langle \xi_2\rangle^{2l}\langle H(\xi_1)\rangle^{2b}}d\xi_1\lesssim 1,
	\end{equation}
	where $B= \left\{  \xi_1;\, \ \   \langle \xi_2\rangle > \frac{1+\iota}{\iota}\langle \xi_1\rangle \right\}$, $\iota$ as defined in \eqref{Eqk}, $\xi_2=\xi-\xi_1$ and 
	\begin{equation}\label{defH}
		H(\xi_1):=\tau-\xi_1^ 3 -\alpha\xi_2^ 3.
	\end{equation}
\end{lemma}
\begin{proof}
	Of course, \eqref{est5} holds when $\alpha<0$ (see the proof of Lemma \ref{lema02}). Thus we will suppose $\alpha>1$. First we will consider $|\xi|>1$ and let
	$$X=\bigl\{\xi_1;\, \ \ |\xi_1|<\iota|\xi|\bigr\},$$
	where $\iota\in (0,\,1)$ will be chosen later, we have
	$$J\leq \underbrace{\int_{B}\frac{\langle \xi_2\rangle}{\langle \xi_2\rangle^{2l}\langle H(\xi_1)\rangle^{2b}}\chi_{_{X}}(\xi_1)d\xi_1}_{J_{1}}+\underbrace{\int_B\frac{\langle \xi_2\rangle}{\langle \xi_2\rangle^{2l}\langle H(\xi_1)\rangle^{2b}}\chi_{_{\R\backslash X}}(\xi_1)d\xi_1}_{J_{2}}.$$
	Starting with $J_{2}$, noting that in $\R\backslash X=\bigl\{\xi_1;\, |\xi_1|\ge \iota|\xi|> \iota\bigr\}$, we have $\langle \xi_2\rangle \leq \frac{1+\iota}{\iota}\langle \xi_1\rangle$, then $B\cap (\R\backslash X)=\varnothing$ and $J_{2}\equiv0$.
	Now, for $J_{1}$, we have
	\begin{equation}\label{est6}
		\langle \xi_1\rangle\lesssim \langle \xi_2\rangle\lesssim \langle\xi\rangle.
	\end{equation}
	We choose $\lambda\in(0,\,3)$ such that 
	\begin{equation}\label{est7}
		H'(\xi_1)=3(\alpha-1)\xi_1^2-6\alpha\xi_1\xi+3\alpha\xi^2\ge \lambda\alpha\xi^2.
	\end{equation}
	So, we have that $H$ is increasing in $X$ and $|H'(\xi_1)|\gtrsim\xi^2\gtrsim_{\alpha,\lambda,c_2} \max\{\xi_1^2,\,\xi_2^2\}$, for all $\xi_1\in X$. Thus, collecting this facts and using Lemma \ref{lema02} and Lemma \ref{lemxm1} item (iv), we get the desired. We determine $\iota\in(0,\,1)$ and $\lambda\in(0,\,3)$ such that \eqref{est7} is valid. If fact, the inequality \eqref{est7} is equivalent to
	\begin{equation}\label{est8}
		\underbrace{3(1-\alpha)\xi_1^2+6\alpha\xi\xi_1}_{\text{quadratic funtion}}\leq \underbrace{(3-\lambda)\alpha\xi^2}_{\text{constant function}}.
	\end{equation}
	So we choose $\iota$ such that the equality  \eqref{est8} is true in the interval $|\xi_1|\leq \iota|\xi|$. Thus, $\iota$ is a root of the quadractic equation
	$$3(1-\alpha)\iota^2\pm 6\alpha \iota+(\lambda-3)\alpha=0.$$
	One can see that 
	\begin{equation}\label{Eqk}
		\iota=(6\alpha-\sqrt{36\alpha^2-12\alpha(1-\alpha)(\lambda-3)} )/[6(\alpha-1)]
	\end{equation}
	belongs to $(0,\,1)$ if $\lambda\in(0,\,3)$.
	
	Now if $|\xi |\leq 1$,  we have
	$$
	\frac12\langle \xi_2 \rangle \leq \langle \xi_1 \rangle \leq 2\langle \xi_2 \rangle
	$$
	and using  Lemma \ref{lemma1.5} we get
	\begin{equation}
		J \leq\int_{\R}\frac{\langle \xi_2\rangle}{\langle \xi_2\rangle^{2l}\langle \tau-\xi_1^ 3 -\alpha\xi_2^ 3\rangle^{2b}}d\xi_1\\
		\lesssim\int_{\R}\frac{\langle \xi_2\rangle}{\langle \xi_1\rangle^{2l}\langle \frac1{\alpha}\tau-\frac1{\alpha}\xi_1^ 3 -\xi_2^ 3\rangle^{2b}}d\xi_1\\
		\lesssim 1.
	\end{equation}
	
\end{proof}
With these Lemmas in hands, we can prove Proposition \ref{bil-LRX-1} and therefore the trilinear estimates in Proposition \ref{prop1}.

\begin{proof}[Proof of the Proposition \ref{bil-LRX-1}]	We only provide a detailed proof of (a), because (b) is analogous. Suppose that $s, k>-\frac12$ and $s-k\leq 1/2$. As before, using Plancherel identity and Cauchy-Schwarz inequality we need to estimate
	\begin{equation}\label{est1}
		\int_{\R^2}\dfrac{\langle  \xi_2 \rangle^{2(s-k)+1}}{\langle \tau_2-\alpha \xi_2^3\rangle^{2b} \langle \xi_1\rangle^{2s} \langle \tau_1-\xi_1^3\rangle^{2b}} d\xi_1 d\tau_1,
	\end{equation}
	where $\xi_2=\xi-\xi_1$, $\tau_2=\tau-\tau_1$, for $\tau$ and $\xi$ fixed and $b=1/2+\epsilon$. Integrating in $\tau_1$ and applying \eqref{lfc1-1}, we need to estimate the following integral
	%
	\begin{equation}
		\mathcal{L}_2=
		\int \limits_{\R} \frac{\langle \xi_2\rangle^{1+2(s-k)}}{\langle\xi_1 \rangle^{2s}\langle H(\xi_1)\rangle^{2b}} d\xi_1,
	\end{equation}
	where $H(\xi_1)$ is given by \eqref{defH}.
	Our goal is to prove $\mathcal{L}_2\lesssim 1$. First we consider the case $s-k\leq 0$, then
	\begin{align}
		\mathcal{L}_2&\lesssim \int_{\R}\frac{\langle \xi_2\rangle}{\langle \xi_1\rangle^{2s}\langle H(\xi_1)\rangle^{2b}}d\xi_1\label{est2}.
	\end{align}
	If $\alpha>1$, we write $\frac{1}{\alpha}H(\xi_1)=\tilde{\tau}-\xi_2^3-\beta\xi_1^3$, where $\tilde{\tau}=\tau/\alpha$ and $\beta=1/\alpha$. Applying \eqref{X1}, with $\beta\in(0,\,1)$ in place of $\alpha$, we get $\mathcal{L}_2\lesssim 1$, remembering that $2b>1-4\epsilon$. On the other hand, if $\alpha<0$ we can see that $H'(\xi_1)=3\alpha\xi_2^2-3\xi_1^2<0$ and $|H'(\xi_1)|\gtrsim_{\alpha}\max\{\xi_1^2,\,\xi_2^2\}$. So, combining Lemma \ref{lema02} and Lemma \ref{lemxm1} item (iv), we get the same bound.
	
	Considering now $0< s-k \leq 1/2$, let
	\begin{equation}\label{setAB}A=\{\xi_1;\, \langle \xi_2\rangle\lesssim \langle \xi_1\rangle\}\ \ \ \ \ \text{and}\ \ \ \ \ B=\{\xi_1;\, \langle \xi_1\rangle\lesssim \langle \xi_2\rangle\}.\end{equation}
	Thus
	$$\mathcal{L}_2\leq \underbrace{\int \limits_{\R} \frac{\langle \xi_2\rangle^{1+2(s-k)}}{\langle\xi_1 \rangle^{2s}\langle H(\xi_1)\rangle^{2b}}\chi_{_{A}}(\xi_1) d\xi_1}_{\mathcal{L}_2^{A}}+\underbrace{\int \limits_{\R} \frac{\langle \xi_2\rangle^{1+2(s-k)}}{\langle\xi_1 \rangle^{2s}\langle H(\xi_1)\rangle^{2b}}\chi_{_{B}}(\xi_1) d\xi_1}_{\mathcal{L}_2^{B}}.$$
	For the first integral, we have 
	$$\mathcal{L}_2^{A}\lesssim \int_{\R}\frac{\langle \xi_2\rangle}{\langle \xi_1\rangle^{2k}\langle H(\xi_1)\rangle^{2b}}d\xi_1,$$
	and this integral  has already been estimated in \eqref{est2}. For the second integral $\mathcal{L}_2^{B}$, similarly as above one can see that if $s\leq 0$,  using Lemma \ref{lemma1.7} we have
	\begin{equation}\label{est3}
		\mathcal{L}_2^{B}\lesssim \int_{B}\frac{\langle \xi_2\rangle}{\langle \xi_2\rangle^{2k}\langle H(\xi_1)\rangle^{2b}}d\xi_1 \lesssim 1.
	\end{equation}
	On the other hand, if $s>0$  again using Lemma \ref{lemma1.7} we have
	\begin{equation}\label{est4}
		\mathcal{L}_2^{B}\lesssim 	\int_{B}\frac{\langle \xi_2\rangle}{\langle \xi_2\rangle^{2(k-s)}\langle H(\xi_1)\rangle^{2b}}d\xi_1\lesssim 1,
	\end{equation}
	since $k-s \geq -1/2$.   
	\noindent
	Now, we prove the same estimate in the range: $s=-1/2$ and $-1/2< k$. Indeed, we need to prove that
	\begin{equation*}
		\|uv\|_{L^2(\R^2)}\lesssim \|u\|_{X^{\alpha}_{k, \frac12+\epsilon}}\|v\|_{X_{-\frac12, \frac12+\epsilon}}.
	\end{equation*}
	This  estimate follows from the estimate \eqref{bil-1.1}, noting that $1/2-2\epsilon<1/2+\epsilon$. 
	
\end{proof}
%
%
%
%
%


\secao{Local well-posedness result: proof of Theorem \ref{loc-sys}}

In view of the previous sections, we are in position to prove the Local well-posedness result given in Theorem \ref{loc-sys}. Here we understand that a solution of the system \eqref{ivp-sy} is in fact a solution of the associated integral equations \eqref{eqint1} and \eqref{eqint2}. We use standard arguments, so we give the proof for the sake of completeness. 
\begin{proof}[Proof of Theorem \ref{loc-sys}]
	Let $s,k\in\R$ such that $s,k>-1/2$ and $|s-k|<1/2$. Let $b=1/2+\epsilon$ and $b'=-1/2+2\epsilon$, with $\epsilon>0$. Let fix $\alpha\in(-\infty,0)\cup(1,\infty)$ (see Remark \ref{obsalpha}). For $a>0$, to be choosen later,  let
	$$B_a=\left\{(v,w)\in X_{s,b}\times X^{\alpha}_{k,b}\,;\,  \|(v,w)\|_{X_{s,b}\times X^{\alpha}_{k,b}}=\|v\|_{X_{s,b}}+\|w\|_{X^{\alpha}_{k,b}}<a  \right\},$$
	a complete metric spaces. For $\delta>0$ (choosen later), we define the map $\mathfrak{F}=\mathfrak{F}_{\delta}:B_a\to X_{s,b}\times X^{\alpha}_{k,b}$ such that $\mathfrak{F}(v,w)=(\mathfrak{F}_1(v,w),\,\mathfrak{F}_2(v,w))$ where
	\begin{align*}
		\mathfrak{F}_1(v,w)&=\eta(t)U(t)\phi-\eta_{\delta}(t)\int_0^tU(t-t')\partial_x(vw^2)(t')dt'\\
		\mathfrak{F}_2(v,w)&=\eta(t)U^{\alpha}(t)\phi-\eta_{\delta}(t)\int_0^tU^{\alpha}(t-t')\partial_x(v^2w)(t')dt'.	
	\end{align*}
	From the linear estimates in Lemma \ref{lemalin} and the trilinear estimates in Proposition \ref{prop1}, for all $(v,w)\in B_a$ we get 
	\begin{align*}\|\mathfrak{F}(v,w)\|_{X_{s,b}\times X^{\alpha}_{k,b}}&\leq C\|(\phi,\psi)\|_{H^s\times H^k}+ C\delta^{1-b+b'}\left\{  \|v\|_{X_{s,b}}\|w\|^2_{X^{\alpha}_{k,b}} + \|v\|_{X_{s,b}}^2\|w\|_{X^{\alpha}_{k,b}}\right\}   \\
		&\leq C\|(\phi,\psi)\|_{H^s\times H^k}+ 2C\delta^{\epsilon}a^3.
	\end{align*}
	Now, taking $a=2C\|(\phi,\psi)\|_{H^s\times H^k}>0$ and $\delta>0$ such that $2C\delta^{\epsilon}<1/2$ we conclude that
	$$\|\mathfrak{F}(v,w)\|_{X_{s,b}\times X^{\alpha}_{k,b}}\leq a, \ \ \ \textrm{for all}\ \ (v,w)\in B_a,$$
	i.e., for $a>0$ and $\delta>0$ as before, $\mathfrak{F}_{\delta}(B_a)\subset B_a$. 
	Also, with a similar arguments and taking $\delta>0$ smaller, if necessary, we can conclude that 
	$$\|\mathfrak{F}(v,w)-\mathfrak{F}(\tilde{v},\tilde{w})\|_{X_{s,b}\times X^{\alpha}_{k,b}}\leq \|(v,w)-(\tilde{v},\tilde{w})\|_{X_{s,b}\times X^{\alpha}_{k,b}},$$
	i.e., $\mathfrak{F}:B_a\to B_a$ is a contraction an has a unique fixed point, establishing  a unique solution $(v,w)$ satisfying \eqref{eqint1} and \eqref{eqint2} for every $t\in[-\delta,\,\delta]$. Because $b>1/2$,  then we have the persistence property $(v,w)\in C([-\delta,\delta]:H^s)\times C([-\delta,\delta]:H^k)$ and also, following a similar arguments, we can conclude that the solution mapping is locally lipchitz from $H^s\times H^k$ into $C([-\delta,\delta]:H^s)\times C([-\delta,\delta]:H^k)$. The uniqueness in the class $X_{s,b}^{\delta}\times X^{\alpha,\delta}_{k,b}$ can be proved with standard arguments (see e.g. \cite{BOP1997} or \cite{Domingues2016}).
\end{proof}
%
%
%
%
%
%

\secao{Ill-posedness results}\label{ill}

In this section we prove some ill-posedness results related to the system \eqref{ivp-sy}.

\subsection{The solution mapping is not $C^3$}
\ \\
\vspace{-0.5cm}

Here we will prove the Theorem \ref{mainTh-ill1.1}. The proof given here follow the structure of the proofs in \protect\cite{Domingues2016} and \protect\cite{DS2019} (see also \protect\cite{Tzvetkov1999}). It is well known that if the LWP results in $H^s(\R)\times H^k(\R) $ for \eqref{ivp-sy} is obtained by means of contraction method, then for a fixed $r>0$ there is a $T=T(r,s,k)>0$ such that the solution mapping 
\begin{eqnarray}
	\label{datasol}
	S\quad\!\! :\quad 
	B_r 
	\quad 
	&\longrightarrow& 
	\quad C
	\bigl(\, [0,T]\ ;\ 
	H^{s} 
	\,\bigr) \times C
	\bigl(\, [0,T]\ ;\ 
	H^{k}\,\bigr 
	)
	\\ 
	\ \ \ \ \ \ (\phi,\psi)\!\! 
	&\mapsto& 
	S_{(\phi,\,\psi)}=(v_{\!_{(\phi,\psi)}}\,,\,w_{\!_{(\phi,\psi)}} )\,,\nonumber 
\end{eqnarray}
is analitic (see Theorem 3 in \protect\cite{BT2006}), where $B_r$ is the $r$-ball centered at the origin of $H^s(\R)\times H^k(\R)$ and $v=v_{\!_{(\phi,\psi)}}$ and $w=w_{\!_{(\phi,\psi)}}$ satisfies, respectively, the integral equations \eqref{eqint1} and \eqref{eqint2} for initial data $\phi$ and $\psi$, in the time interval $[0,\,T]$. In this way, if we show that for a certain indices $(s,\,k)$ the solution mapping is not three times differentiable at the origin $(0,\,0)$ for all $T>0$ fixed,  the contraction method can not be applied to get LWP for these indices $(s,\,k)$.

Fixing $t\in[0,\,T]$, 
we define the \textit{flow mapping} associated to the system \eqref{ivp-sy} the  map 
\begin{eqnarray}
	\label{fluxo}
	S^{\,t}\!\quad\!\! :\quad 
	B_r
	\quad 
	&\longrightarrow& 
	\ \ \,H^s(\R)\times H^k(\R)
	\\ 
	\ \ \ \ \ \ (\phi,\psi)\!\! 
	&\mapsto& 
	S^t_{(\phi,\,\psi)}=S_{(\phi,\,\psi)}(t)=\bigl(v_{\!_{(\phi,\psi)}}(t)\,,\,w_{\!_{(\phi,\psi)}}(t)\bigr).\nonumber 
\end{eqnarray}

The Theorem \ref{mainTh-ill1.1} follows from the next proposition:
\begin{proposition}\label{ill-pos} Assume that the system \eqref{ivp-sy} is locally well-posed in the time interval $[0,\,T]$. Then we have 
	\begin{enumerate}
		\item[(a)] The solution mapping \eqref{datasol} is not 3-times Fréchet differentiable at the origin in $H^s(\R)\times H^k(\R)$ if $|s-k|>2$.
		
		\item[(b)] The flow mapping \eqref{fluxo} is not 3-times Fréchet differentiable at the origin in $H^s(\R)\times H^k(\R)$ if $s<-1/2$ or $k<-1/2$.
	\end{enumerate}
	
\end{proposition}

\begin{remark} We point out that the result in (b) implies that the solution mapping is not 3-times differentiable at $(0,\,0)$ for the same indices $s$ and $k$. For a more detailed discussion we refer Remarks 1.4 and 1.5 in \cite{DS2019}.
\end{remark}
Before we start to prove our results, we need to do some calculations.
If $S^t$ is 3-times Fréchet differentiable at the origin in $H^s(\R)\times H^k(\R)$, then its third derivative $D^3S^t_{(0,\,0)}$ belongs to $\mathcal{B_1}$, the normed space of bounded trilinear applications from $(H^s\times H^k)\times (H^s\times H^k)\times (H^s\times H^k)$ to $H^s\times H^k$ and 
we have the following estimate for the third Gâteaux derivative of $S^{\,t}\!$ 
\begin{align}
	\left\| \frac{\partial^3 S^{\,t}_{(0,0)}}{\partial\Phi_0\partial\Phi_1\partial\Phi_2}\right\|_{_{H^{s}\times H^k}}\!\! 
	&=
	\left\| D^3S^{\,t}_{(0,0)}(\Phi_0,\Phi_1,\Phi_2)\right\|_{_{H^{k}\times H^l}}\nonumber\\
	&\leq\! 
	\left\|  D^3S^{\,t}_{(0,0)}  \right\|_{_{\mathcal B_1}} \|\Phi_0\|_{_{H^{k}\times H^l}}\|\Phi_1\|_{_{H^{k}\times H^l}}\|\Phi_2\|_{_{H^{k}\times H^l}}\,,
	\ \forall\Phi_0,\Phi_1,\Phi_2\in H^{k}\times H^l.\label{ill01}
\end{align}
Also, if $S$ is 3-times Fréchet differentiable at the origin, we have a similar estimate:
\begin{align}
	\sup_{t\in [0,\,T]}\left\| \frac{\partial^3 S^{\,t}_{(0,0)}}{\partial\Phi_0\partial\Phi_1\partial\Phi_2}\right\|_{_{H^{s}\times H^k}}\!\! 
	&=
	\left\| D^3S_{(0,0)}(\Phi_0,\Phi_1,\Phi_2)\right\|_{_{H^{k}\times H^l}}\nonumber\\
	&\leq\!\! 
	\left\|  D^3S_{(0,0)}  \right\|_{_{\mathcal B_2}}\!\! \|\Phi_0\|_{_{H^{k}\times H^l}}\!\|\Phi_1\|_{_{H^{k}\times H^l}}\!\|\Phi_2\|_{_{H^{k}\times H^l}}\,,
	\forall\Phi_0,\Phi_1,\Phi_2\in H^{k}\!\times\! H^l,\label{ill02}
\end{align}
where $\mathcal{B_2}$ is the normed space of bounded trilinear applications from $(H^s\times H^k)\times (H^s\times H^k)\times (H^s\times H^k)$ to $C(\, [0,T]\ ; 
H^{s} \,) \times C(\, [0,T] ;\ H^{k}\,)$. For $\Phi_k=(\phi_k,\,\psi_k)\in \mathcal{S}(\R)\times\mathcal{S}(\R)$, $k=0,1,2$, we can calculate the third Gâteaux derivative of each component of $S^t$:
\begin{align*}
	\frac{\partial^3 v_{(0,0)}}{\partial\Phi_0\partial\Phi_1\partial\Phi_2}&=-2\int_0^tU(t-t')\partial_x\bigl\{U(t')\phi_0U^{\alpha}(t')\psi_1U^{\alpha}(t')\psi_2 + U(t')\phi_1U^{\alpha}(t')\psi_2U^{\alpha}(t')\psi_0 \\
	& \ \ \ \ + U(t')\phi_2U^{\alpha}(t')\psi_1U^{\alpha}(t')\psi_0\bigr\}dt'
\end{align*}
and
\begin{align*}
	\frac{\partial^3 w_{(0,0)}}{\partial\Phi_0\partial\Phi_1\partial\Phi_2}&=-2\int_0^tU^{\alpha}(t-t')\partial_x\bigl\{U(t')\phi_0U(t')\phi_1U^{\alpha}(t')\psi_2 + U(t')\phi_0U(t')\phi_2U^{\alpha}(t')\psi_1 \\
	& \ \ \ \ + U(t')\phi_1U(t')\phi_2U^{\alpha}(t')\psi_0\bigr\}dt.
\end{align*}
So, for directions $\Phi_0=(\phi_0,\,0)$, $\Phi_1=(0,\,\psi_1)$ and $\Phi_2=(0,\,\psi_2)$ in $\mathcal{S}(\R)\times \mathcal{S}(\R)$ we get
\begin{align}
	\frac{\partial^3 v_{(0,0)}}{\partial\Phi_0\partial\Phi_1\partial\Phi_2}&=-2\int_0^tU(t-t')\partial_x\Bigl\{U(t')\phi_0 U^{\alpha}(t')\psi_1U^{\alpha}(t')\psi_2 \Bigr\}dt'\label{ill03}
\end{align}
and for directions $\Phi_0=(\phi_0,\,0)$, $\Phi_1=(\phi_1,\,0)$ and $\Phi_2=(0,\,\psi_2)$
\begin{align}
	\frac{\partial^3 w_{(0,0)}}{\partial\Phi_0\partial\Phi_1\partial\Phi_2}&=-2\int_0^tU^{\alpha}(t-t')\partial_x\Bigl\{U(t')\phi_0U(t')\phi_1U^{\alpha}(t')\psi_2 \Bigr\}dt'.\label{ill04}
\end{align}
\vspace{0.3cm}

With these in hands we also need a elementary result, proved  in \protect\cite{Domingues2016}:

\begin{lemma}\label{lemleandro} Let $A$, $B$, $R$ Lebesgue-measurable subsets of $\R^n$ such that\footnote{Here $X-Y=\{x-y;\, x\in X \ \textrm{and}\ y\in Y\}$.} $R-B\subset A$. Then\footnote{$|X|$ denotes de Lebesgue measure of the set $X$}
	$$\|\chi_{_{A}}\ast \chi_{_{B}}\|_{L^2(\R)}\gtrsim |B|  |R|^{1/2}.$$
\end{lemma}
The following lemma, which is a version of the elementary lemma above, plays a central role in the proof of the Theorem \ref{ill-pos}.
\begin{lemma}\label{lemaelem} Let $A$, $B$, $C$, $R$ Lebesgue-measurable  subsets of $\R^n$ such that $R-B-C\subset A$. Then
	$$\|\chi_{_{A}}\ast \chi_{_{B}}\ast \chi_{_{C}}\|_{L^2(\R)}\gtrsim |B||C||R|^{1/2}.$$
\end{lemma}

Now we move to the

\begin{proof}[Proof of Proposition \ref{ill-pos}]
	(a)	Suppose that $S^t$ is 3-times differentiable at the origin in $H^s\times H^k$. Because \eqref{ill03} and  \eqref{ill04} are the components of the third Gâteaux derivative at the origin,  we have  the same estimate  with the $H^s\times H^k$ norm of \eqref{ill03} or  \eqref{ill04} in the right-hand side of \eqref{ill02}. We start with the first component.  Considering  
	$A,B,C\subset\R$ bounded subsets and choosing $\phi_0,\psi_1,\psi_2\in\mathcal S(\R)$ such that\footnote
	{\ For $B$, bounded subset of $\R$, $\langle \cdot\rangle^l\widehat{\phi}\sim \chi_{_{B}}$ means
		$\chi_{_{B}}\leq\langle\cdot\rangle^l\,\widehat{\varphi}$\, with\, $\|\varphi\|_{H^l}\leq2\|\chi_{_{B}}\|_{L^2}$ 
	}
	$\langle\cdot\rangle^s\,\widehat{\phi_0}\sim \chi_{_{A}}$, $\langle\cdot\rangle^k\,\widehat{\psi_1} \sim \chi_{_{B}}$ and $\langle\cdot\rangle^k\,\widehat{\psi_2} \sim \chi_{_{C}}$, 
	we have that\footnote{\ $(*.*)_R\ $(or $(*.*)_L$) denotes the right(or left)-hand side of an equality or inequality numbered by $(*.*)$}
	\begin{align}
		\eqref{ill02}_{L}&\gtrsim  \left\| \int_0^t\frac{\langle \xi\rangle^s|\xi|}{\langle \xi_1 \rangle^s\langle\xi_2  \rangle^k\langle\xi_3\rangle^k} \cos(t'Q_{\alpha})\chi_{_{A}}(\xi_1)\chi_{_{B}}(\xi_2)\chi_{_{C}}(\xi_3)  d\xi_1d\xi_2dt'    \right\|_{L^2_{\xi}}\label{ill05}
	\end{align}
	where  $\xi_3:=\xi-\xi_1-\xi_2$ and 
	\begin{equation}\label{Q}
		Q_{\alpha}=Q_{\alpha}(\xi,\,\xi_1,\,\xi_2):=\xi^3-\xi_1^3-\alpha\xi_2^3-\alpha\xi_3^3.
	\end{equation} 
	So combining \eqref{ill05} with \eqref{ill02} we get 
	\begin{equation}\label{ill06}
		\sup_{t\in[0,\,T]}\left\| \int_0^t\frac{\langle \xi\rangle^s|\xi|}{\langle \xi_1 \rangle^s\langle\xi_2  \rangle^k\langle\xi_3\rangle^k} \cos(t'Q_{\alpha})\chi_{_{A}}(\xi_1)\chi_{_{B}}(\xi_2)\chi_{_{C}}(\xi_3)  d\xi_1d\xi_2dt'    \right\|_{L^2_{\xi}}\lesssim |A|^{1/2}|B|^{1/2}|C|^{1/2},
	\end{equation}
	for all $t\in [0,\,T]$. 
	In view of Lemma \ref{lemaelem}, now we must choose the sets $A$, $B$, $C$, $R$ and a sequence of times $t_N\in [0,\,T]$ in this way: for $N\in\N$ 
	\begin{equation*}\label{ABC*}
		A_N=\{\xi_1\in \R : |\xi_1|<1/2\}, 
		\quad\quad 
		B_N=\{\xi_2\in \R : |\xi_2|<1/4\}, \quad\quad C_N= \{\xi_3\in \R : |\xi_3-N|<1/8\}
	\end{equation*}
	\begin{equation*}\label{tn*}
		R_N=\{\xi\in \R : |\xi-N|<1/8\}  \quad\quad\textrm{and}\quad\quad 
		t_N=\frac {T}{2N^3(1+T)}.
	\end{equation*} 
	We have  $R_N-B_N-C_N\subset A_N$ and $t_N\in[0,\,T]$. 
	Because $\xi_1\in A_N,\,\xi_2\in B_N$ and $\xi_3\in C_N$, thus
	\begin{equation*}\label{ill07} 
		\frac{\langle \xi\rangle^s|\xi|}{\langle \xi_1 \rangle^s\langle\xi_2  \rangle^k\langle\xi_3\rangle^k}\ \sim\ N^{s-k+1} 
		\quad\quad\textrm{and}\quad\quad 
		\cos(t'Q_{\alpha})>1/2\,, 
		\ \ 
		\forall t'\in[0,t_N]. 
	\end{equation*} 
	From Lemma \ref{lemaelem} and \eqref{ill06} yields
	\begin{equation}\label{ill08} 
		t_N|R_N|^{\frac12}|B_N||C_N|\,N^{s-k+1}
		\  \lesssim\ 
		|A_N|^{1/2}|B_N|^{1/2}|C_N|^{1/2}, 
		\quad 
		\forall N\in\N.  
	\end{equation}
	Taking account that $|A_N|\sim|B_N|\sim|C_N|\sim |R_N|\sim 1$ and $t_N\sim N^{-3}$, then we have $s-k<2$.
	
	Now, dealing with the second component \eqref{ill04}, analogously as we did before, we have
	\begin{equation}\label{ill09}
		\sup_{t\in[0,\,T]}\left\| \int_0^t\frac{\langle \xi\rangle^k|\xi|}{\langle \xi_1 \rangle^s\langle\xi_2  \rangle^s\langle\xi_3\rangle^k} \cos(t'P_{\alpha})\chi_{_{A}}(\xi_1)\chi_{_{B}}(\xi_2)\chi_{_{C}}(\xi_3)  d\xi_1d\xi_2dt'    \right\|_{L^2_{\xi}}\lesssim |A|^{1/2}|B|^{1/2}|C|^{1/2},
	\end{equation}
	for all $t\in [0,\,T]$, where \begin{equation}\label{P}
		P_{\alpha}=P_{\alpha}(\xi,\,\xi_1,\,\xi_2)=\xi^3-\xi_1^3-\xi_2^3-\alpha\xi_3^3.
	\end{equation}
	Interchanging the rules of the sets $B_N$ and $C_N$ given before we can conclude that $k-s<2$. We finishes the proof of item (a).
	
	\hspace{-0.5cm}(b) In this case we are not able to take a sequence $t_N\to 0$, because $t$ is fixed number. We need to control de argument in the terms $\cos(t'Q_{\alpha})$ and $\cos(t'P_{\alpha})$, and we get this by making both $Q_{\alpha}$ and $P_{\alpha}$ suficiently small.  Fixing $t\in[0,\,T]$, by \eqref{ill01}, similarly what we did to get \eqref{ill06}, we have that
	\begin{equation}\label{ill10}
		\left\| \int_0^t\frac{\langle \xi\rangle^s|\xi|}{\langle \xi_1 \rangle^s\langle\xi_2  \rangle^k\langle\xi_3\rangle^k} \cos(t'Q_{\alpha})\chi_{_{A}}(\xi_1)\chi_{_{B}}(\xi_2)\chi_{_{C}}(\xi_3)  d\xi_1d\xi_2dt'    \right\|_{L^2_{\xi}}\lesssim |A|^{1/2}|B|^{1/2}|C|^{1/2},
	\end{equation}
	For $N\in\N$, considering the sets
	\begin{equation*}\label{AB**}
		A_N=\{\xi_1\in \R : |\xi_1-aN|<\varepsilon \langle t\rangle^{-1} N^{-2}\}, 
		\quad\quad 
		B_N=\{\xi_2\in \R : |\xi_2-bN|<\varepsilon \tfrac14 \langle t\rangle^{-1} N^{-2}\},  
	\end{equation*}
	\begin{equation*}\label{C**}
		C_N= \{\xi_3\in \R : |\xi_3-cN|<\varepsilon \tfrac14 \langle t\rangle^{-1} N^{-2}\} \quad\quad \textrm{and}\quad\quad R_N=\{\xi\in \R : |\xi-N|<\varepsilon \tfrac14 \langle t\rangle^{-1} N^{-2}\}  
	\end{equation*}
	where the constant $\varepsilon>0$ will be small, but fixed, and the positive constants $a$, $b$ and $c$ satisfies the conditions
	\begin{equation}\label{constcond01}
		\left\{\begin{array}{l}
			a+b+c=1,\\
			a^3+\alpha b^3+\alpha c^3=1.
		\end{array}\right.
	\end{equation}
	First we note that $R_N-B_N-C_N\subset A_N$ and $|A_N|\sim|B_N|\sim|C_N|\sim|R_N|\sim N^2$. Also, for $\xi_1\in A_N$, $\xi_2\in B_N$, $\xi_3\in C_N$ yields
	$$\frac{\langle \xi\rangle^s|\xi|}{\langle \xi_1 \rangle^s\langle\xi_2  \rangle^k\langle\xi_3\rangle^k}\sim N^{1-2k}$$
	and 
	\begin{align*}
		|Q_{\alpha}|&=|\xi^3-N^3+N^3-\xi_1^3-\alpha\xi_2^3-\alpha\xi_3^3|\leq |\xi^3-N^3|+|(a^3+\alpha b^3+\alpha c^3)N^3-\xi_1^3-\alpha\xi_2^3-\alpha\xi_3^3\\
		& \leq  |\xi^3-N^3|+|(aN)^3-\xi_1^3|+|\alpha||(bN)^3-\xi_2^3|+|\alpha||(cN)^3-\xi_3^3|\\
		&\lesssim_{_{a,b,|\alpha|}}\ \langle t'\rangle^{-1}\varepsilon, 
	\end{align*}
	where we use the elementary identity $A^3-B^3=(A-B)(A^2+AB+B^2)$. This implies the desired estimate $$\cos(t'Q_{\alpha})>1/2, \ \ \textrm{for all}\ \ t'\in [0,\,t].$$
	From the Lemma \ref{lemaelem} and \eqref{ill10} we get
	$$t(N^{-2})^{1/2}N^{-2}N^{-2}N^{1-2k}\lesssim (N^{-2})^{1/2}(N^{-2})^{1/2}(N^{-2})^{1/2},$$
	and so $k\geq -1/2$.
	For the second component \eqref{ill04}, in the same way as before, we have for a fixed $t$ in $[0,\,T]$ the following inequality
	\begin{equation}\label{ill11}
		\left\| \int_0^t\frac{\langle \xi\rangle^k|\xi|}{\langle \xi_1 \rangle^s\langle\xi_2  \rangle^s\langle\xi_3\rangle^k} \cos(t'P_{\alpha})\chi_{_{A}}(\xi_1)\chi_{_{B}}(\xi_2)\chi_{_{C}}(\xi_3)  d\xi_1d\xi_2dt'    \right\|_{L^2_{\xi}}\lesssim |A|^{1/2}|B|^{1/2}|C|^{1/2}.
	\end{equation}
	Of course, with the same choice of subsets $A_N$, $B_N$, $C_N$, etc. but with
	\begin{equation}\label{constcond02}
		\left\{\begin{array}{l}
			a+b+c=1,\\
			a^3+b^3+\alpha c^3=1,
		\end{array}\right.
	\end{equation}
	we can conclude that $s\geq -1/2$ and we finish the proof of  the item (b) and also the proof of the Theorem.
\end{proof}
%
%
%
%
%
%
%
\subsection{Failure of trilinear estimates}
\ \\
\vspace{-0.5cm}

In this subsection we will prove the failure of the trilinear estimates \eqref{tlint-m1} and \eqref{tlint-m2}. 

\begin{proof}[Proof of Proposition \ref{prop1ill}]
	We prove only item (a), because the item (b) follows analogously. Using definition of the $X_{s,b}$-norm  and Plancherel's identity,  the estimate \eqref{tlint-m1} is equivalent to 
	\begin{equation}\label{tril-fail3}
		\| \mathcal{T}_s(f,g,h) \|_{L^2_{\xi} L^2_{\tau}} \lesssim \|f\|_{L^2(\R^2)} \|g\|_{L^2(\R^2)}\|h\|_{L^2(\R^2)},
	\end{equation}
	where 
	\begin{equation}\label{trill-f4}
		\mathcal{T}_s(f,g,h):= \Big\|\langle\xi\rangle^s\langle\tau-\xi^3\rangle^{b'}\!\xi\!\int_{_{\!\R^4}}\! \dfrac{\widetilde{f}(\xi_1, \tau_1)\widetilde{g}(\xi_2, \tau_2) \widetilde{h}(\xi_3, \tau_3)}{\langle \xi_1 \rangle^s\langle \tau_1-\xi_1^3 \rangle^{b} \langle \xi_2 \rangle^k\langle \tau_2-{\alpha}\xi_2^3 \rangle^{b}\langle \xi_3 \rangle^k\langle \tau_3-\alpha\xi_3^3 \rangle^{b}}d\xi_1d\xi_2d\tau_1d\tau_2\Big\|_{L^2_{\xi\tau}},
	\end{equation}
	with $\xi_3=\xi-\xi_1-\xi_2$, $\tau_3=\tau-\tau_1-\tau_2$.
	
	Now, suppose that $s-2k>1$.	Let $c_2$ and $c_3$ be two constants satisfying
	\begin{equation}\label{cond-const}
		\begin{split}
			&c_2+c_3=1,\\
			&\alpha c_2^3+{\alpha}c_3^3 =1,
		\end{split}
	\end{equation}
	and let  $\sigma_1=\tau_1-\alpha\xi_1^3$, $\sigma_2=\tau_2-\alpha\xi_2^3$, $\sigma_3=\tau_3-\alpha\xi_3^3$  in order to  apply the Lemma \ref{lemaelem}, we define the sets
	\begin{equation*}
		A_N=\{(\xi_1, \tau_1) ;\ |\xi_1|< N^{-2},\ |\sigma_1|<C_\alpha\},  \quad B_N=\{(\xi_2, \tau_2) ;\ |\xi_2-c_2 N|< N^{-2}/3,\ |\sigma_2|<1\},
	\end{equation*}
	\begin{equation*}
		C_N=\{(\xi_3, \tau_3) ;\ |\xi_3-c_3 N|< N^{-2}/3,\  |\sigma_3|<1\} \ \ \  \textrm{and}\ \ \ R_N=\{(\xi, \tau) ;\ |\xi-N|< N^{-2}/3,\  |\tau- \xi^{3}|<1\}.
	\end{equation*}
	Then $R_N-B_N-C_N \subset A_N$. In fact , if $(\xi_1, \tau_1)=(\xi, \tau)-(\xi_2, \tau_2)-(\xi_3, \tau_3)$ with $(\xi, \tau) \in R_N$, $(\xi_2, \tau_2) \in B_N$ and $(\xi_3, \tau_3) \in C_N$, then using \eqref{cond-const}, we have
	$$
	|\xi_1|=|\xi-\xi_2-\xi_3| \leq |\xi-N|+|N-\xi_2-\xi_3| \leq \frac{1}3N^{-2}+|c_2N-\xi_2|+|c_3N-\xi_3| < N^{-2}.
	$$
	On the other hand using again \eqref{cond-const}
	\begin{equation*}
		\begin{split}
			|\sigma_1| =&|\tau_1 -\xi_1^3|= |\tau -\tau_2-\tau_3 -\xi_1^3|= |\tau-\xi^3+\xi^3-\sigma_2-\alpha\xi_2^3- \sigma_3-\alpha\xi_3^3-\xi_1^3|\\
			\leq & |\tau-\xi^3|+|\sigma_2|+|\sigma_3|+ |\xi_1|^3+|\xi^3-\alpha\xi_2^3-\alpha\xi_3^3|\\
			\leq & 3+N^{-6}+|\xi^3-N^3+\alpha N^3c_2^3 -\alpha\xi_2^3+\alpha N^3c_3^3-\alpha\xi_3^3|\\
			\lesssim & (1+|\alpha|).
		\end{split}
	\end{equation*}
	Consider $\widetilde{f}=\chi_{A_N}$, $\widetilde{g}=\chi_{B_N}$, $\widetilde{h}=\chi_{C_N}$. From the definition of $A_N$, $B_N$ and $C_N$ holds that $\langle  \xi_1\rangle \sim 1$, $\langle  \xi_2\rangle \sim N$,  $\langle  \xi_3\rangle \sim N$, $\langle  \xi \rangle \sim N$ and $\langle \sigma_j \rangle \sim 1$, $j=1,2,3$. By \eqref{tril-fail3} and Lemma \ref{lemaelem} we obtain
	\begin{equation*}
		\begin{split}
			N^{s-2k+1} \|R_N\|_{L^2(\R^2)} \|\chi_{B_N}\|_{L^1(\R^2)} \|\chi_{C_N}\|_{L^1(\R^2)}\leq& N^{s-2k+1} \|\chi_{A_N} \ast \chi_{B_N}\ast \chi_{C_N}\|_{L^2(\R^2)}\\
			\leq & \|\chi_{A_N} \|_{L^2(\R^2)}\| \chi_{B_N}\|_{L^2(\R^2)}\| \chi_{C_N}\|_{L^2(\R^2)}\end{split}
	\end{equation*}
	which implies
	$$
	NN^{s-2k} N^{-1} N^{-4} \lesssim N^{-3}
	$$
	and the last inequality is false if $s-2k>1$.

	Finally, if $k<-\frac12$, let $c_1$, $c_2$ and $c_3$ be three constants satisfying
	\begin{equation}\label{cond-constx}
		\begin{split}
			&c_1+c_2+c_3=1,\\
			& c_1^3+\alpha c_2^3+\alpha c_3^3 =1.
		\end{split}
	\end{equation}
	
	Analogously as above, if we consider the sets
	\begin{equation*}
		A_N=\{(\xi_1, \tau_1) ;\quad |\xi_1-c_1N|< N^{-2}, |\sigma_1|<C_{\alpha}\}, \quad  \quad B_N=\{(\xi_2, \tau_2) ;\quad |\xi_2-c_2N|< N^{-2}/3, |\sigma_2|<1\},
	\end{equation*}
	\begin{equation*}
		C_N=\{(\xi_3, \tau_3) ;\quad |\xi_3-c_3N|<  N^{-2}/3, |\sigma_3|<1\} \quad \textrm{and} \quad R_N=\{(\xi, \tau) ;\quad |\xi-N|< N^{-2}/3, |\tau-\xi^3|<1\}.
	\end{equation*}
	Then using the condition \eqref{cond-constx} we obtain
	$$
	|\xi_1-c_1N|=|\xi-\xi_2-\xi_3-c_1N| \leq |\xi-N|+|c_2N-\xi_2|+|c_3N-\xi_3| < N^{-2},
	$$
	as above and using the condition \eqref{cond-constx} again
	\begin{equation*}
		\begin{split}
			|\sigma_1| \leq & |\tau-\xi^3|+|\sigma_2|+|\sigma_3|+ |\xi^3-\xi_1^3-\alpha\xi_2^3-\alpha\xi_3^3|\\
			\lesssim & 3+|\xi^3-N^3|+ |c_1^3N^3- \xi_1^3|+|\alpha|\,|c_2^3N^3- \xi_2^3| +|\alpha|\,|c_3^3N^3- \xi_3^3|\\
			\lesssim & (1+|\alpha|).
		\end{split}
	\end{equation*}
	Thus $R_N-B_N-C_N \subset A_N$ and $\langle  \xi_j\rangle \sim N$, $j=1,2,3$, $\langle  \xi \rangle \sim N$ and $\langle \sigma_j \rangle \sim 1$, $j=1,2,3$. Again, by \eqref{tril-fail3} and Lemma \ref{lemaelem} we get
	\begin{equation*}
		\begin{split}
			\dfrac{N^{1+s}}{N^{s}N^{2k} }\|R_N\|_{L^2(\R^2)} \|\chi_{B_N}\|_{L^1(\R^2)} \|\chi_{C_N}\|_{L^1(\R^2)}\leq& \dfrac{N^{1+s}}{N^{s}N^{2k} } \|\chi_{A_N} \ast \chi_{B_N}\ast \chi_{C_N}\|_{L^2(\R^2)}\\
			\leq & \|\chi_{A_N} \|_{L^2(\R^2)}\| \chi_{B_N}\|_{L^2(\R^2)}\| \chi_{C_N}\|_{L^2(\R^2)}\end{split}
	\end{equation*}
	therefore
	$$
	NN^{-2k} N^{-1} N^{-4} \lesssim N^{-3}
	$$
	and the last inequality is false if $-2k>1$.\end{proof}

\begin{proof}[Proof of Proposition \ref{prop1illx1}]
	We only give the proof of (a), because (b) can be proved in the same way. We will consider the following sets
	\begin{equation*}
		A_N=\{(\xi_1, \tau_1) ;\quad |\xi_1|< 1, |\sigma_1| \sim N^3\}, \quad  \quad B_N=\{(\xi_2, \tau_2) ;\quad |\xi_2|< \frac12, |\sigma_2|<1\},
	\end{equation*}
	\begin{equation*}
		C_N=\{(\xi_3, \tau_3) ;\quad |\xi_3-N|<\frac14, |\sigma_3|<1\} \quad \textrm{and} \quad R_N=\{(\xi, \tau) ;\quad |\xi-N|<\frac14, |\tau-\xi^3|<1\}.
	\end{equation*}
	Then $R_N-B_N-C_N \subset A_N$ and $\langle  \xi_3\rangle \sim N$, $\langle  \xi_j \rangle \sim 1,  j=1,2$, $\langle  \xi \rangle \sim N$ and $\langle \sigma_j \rangle \sim 1$, $j=2,3$. Using \eqref{trill-f4} we obtain
	\begin{equation*}
		\begin{split}
			\dfrac{N^{1+s}}{N^{3b}N^{k} }\|R_N\|_{L^2(\R^2)} \|\chi_{B_N}\|_{L^1(\R^2)} \|\chi_{C_N}\|_{L^1(\R^2)}\leq& \dfrac{N^{1+s}}{N^{3b}N^{k} } \|\chi_{A_N} \ast \chi_{B_N}\ast \chi_{C_N}\|_{L^2(\R^2)}\\
			\leq & \|\chi_{A_N} \|_{L^2(\R^2)}\| \chi_{B_N}\|_{L^2(\R^2)}\| \chi_{C_N}\|_{L^2(\R^2)}\end{split}
	\end{equation*}
	therefore
	$$
	\dfrac{N^{1+s}}{N^{3b}N^{k} } \lesssim N^{3/2}
	$$
	and the last inequality is false if $s-k>1/2+3b$, $b=1/2+\epsilon$.
	
\end{proof}

For the next proof we will use another elementary result, proved in \cite{CP2019}.
\begin{lemma}\label{conv-lema2}
	Let $R_j:=[a_j, b_j]\times[c_j, d_j]\subset\mathbb{R}^2$, $j=1,\cdots,n$ be rectangles such that $ b_j-a_j=N$ and $d_j-c_j=M$. Then
	\begin{equation}\label{conv-e2}
		\|\chi_{R_1}*\chi_{R_2}*\cdots*\chi_{R_n}\|_{L^2(\mathbb{R}^2)}\sim (NM)^{n-\frac12}=|R_j|^{n-\frac12}.
	\end{equation}
\end{lemma}
\vspace{0.5cm}

\begin{proof}[Proof of Proposition \ref{trilinendpoint}]
	We only prove that \eqref{x2.3 1} fails to hold. Using Plancherel's identity, 
	the estimate \eqref{x2.3 1} is equivalent to showing that
	\begin{equation}\label{xfail-be2}
		\begin{split}
			\mathcal{B}_s&:=\Big\| \langle \xi \rangle^{-\frac12}\langle \tau- \xi^3 \rangle^{-\frac12+4\epsilon}\!\int_{\R^4}\! \dfrac{\langle \xi_3 \rangle^{1/2}\langle \xi_2 \rangle^{1/2}\langle \xi_1 \rangle^{1/2}\tilde{f}(\xi_1, \tau_1)\tilde{g}(\xi_2, \tau_2)\tilde{h}(\xi_3, \tau_3) }{\langle \tau_1- \alpha\xi_1^3 \rangle^{\frac12+\epsilon} \langle \tau_2-\alpha\xi_2^3 \rangle^{\frac12-2\epsilon} \langle \tau_3-\xi_3^3 \rangle^{\frac12-2\epsilon}} d\xi_1 d\tau_1d\xi_2 d\tau_2 \Big\|_{L^2_{\xi\tau}} \\
			&\le \|f\|_{L^2(\R^2)} \|g\|_{L^2(\R^2)} \|h\|_{L^2(\R^2)},
		\end{split}
	\end{equation}
	where 
	$\widetilde{f}(\xi,\tau)=\langle\xi\rangle^{-\frac12}\langle\tau-\alpha\xi^3\rangle^{\frac12-2\epsilon}\widetilde{w_1}(\xi,\tau)$,  $\widetilde{g}(\xi,\tau)=\langle\xi\rangle^{-\frac12}\langle\tau-\alpha \xi^3\rangle^{\frac12+\epsilon} \widetilde{w_2}(\xi,\tau)$, $\widetilde{h}(\xi,\tau)=\langle\xi\rangle^{-\frac12}\langle\tau-\xi^3\rangle^{\frac12+\epsilon} \widetilde{v}(\xi,\tau)$, $\xi_1+\xi_2+\xi_3=\xi$ and $\tau_1+\tau_2+\tau_3=\tau$.

	We will construct functions $f$, $g$ and $h$ for which the estimate \eqref{xfail-be2} fails to hold for all $\epsilon>0$. 	Let $c_1$, $c_2$ and $c_3$ three numbers such that
	$$
	c_1+c_2+c_3=1, \quad \alpha c_1^3+\alpha c_2^3+ c_3^3=1,
	$$
	and consider two rectangles $R_1$, $R_2$ and  $R_3$  with centres respectively at $(c_1N, \alpha(c_1N)^3)$, $(c_2N, \alpha(c_2N)^3)$ and  $(c_3N, (c_3N)^3)$, and each with dimension $N^{-(2+r)}\times N^{-r}$, where $-2<r<0$. Now, consider $f$ and $g$ defined, via their Fourier transform, by	$\widetilde{f} = \chi_{R_1},  $ $\widetilde{g}= \chi_{R_2}$,  $ \widetilde{h}= \chi_{R_3}$.
	It is easy to see that 
	\begin{equation}\label{xfg-norms}
		\|f\|_{L^2(\R^2)} =\|g\|_{L^2(\R^2)}=\|h\|_{L^2(\R^2)}=N^{-(1+r)}.
	\end{equation}
	Also, 
	\begin{equation}\label{xfail-be5}
		| \xi_1-c_1N|\leq \frac{1}2N^{-(2+r)}, \qquad |\tau_1-\alpha (c_1N)^3|\leq \frac{1}2N^{-r},
	\end{equation}
	\begin{equation}\label{xfail-be6}
		|\xi_2-c_2 N|\leq \frac{1}2N^{-(2+r)}, \qquad |\tau_2-\alpha (c_2N)^3|\leq \frac{1}2N^{-r}
	\end{equation}
	and
	\begin{equation}\label{x1fail-be6}
		|\xi_3-c_3 N|\leq \frac{1}2N^{-(2+r)}, \qquad |\tau_3- (c_3N)^3|\leq \frac{1}2N^{-r}.
	\end{equation}
	We have $|\xi_j|\sim N$ and
	\begin{equation*}
		\begin{split}
			| \xi-N|&=| \xi_1+\xi_2+\xi_3-c_1N- c_2N-c_3N |\leq | \xi_1-c_1N |+| \xi_2-c_2N |+| \xi_3-c_3N |
			\leq \frac{3}{2}N^{-(2+r)}.
		\end{split}
	\end{equation*}
	Also, one can prove  that 
	\begin{equation*}
		\begin{split}
			|\tau-  \xi^3|=&|\tau_1+\tau_2+\tau_3-  (\alpha c_1^3+\alpha c_2^3+ c_3^3)N^3+  N^3- \xi^3|\\
			\leq &|\tau_1-\alpha (c_1N)^3|+|\tau_2-\alpha (c_2N)^3| +|\tau_3-  (c_3 N)^3|+\,|N^3- \xi^3|\\
			\leq &  \frac{3}2N^{-r}+ \,\frac{1}2N^{-(2+r)}|N^2+N\xi+ \xi^2|
			\lesssim   \,N^{-r},
		\end{split}
	\end{equation*}
	and for $j=1,2$:
	\begin{equation*}
		\begin{split}
			|\tau_j- \alpha \xi_j^3|=&|\tau_j- \alpha(c_jN)^3+ \alpha(c_jN)^3-\alpha \xi_j^3|\leq  |\tau_j-  \alpha(c_jN)^3|+\,|\alpha||(c_jN)^3- \xi_j^3|\\
			\leq &  \frac{1}2N^{-r}+ \frac{1}2N^{-(2+r)}|c_j^2N^2+c_jN\xi_j+ \xi_j^2|
			\lesssim   \,N^{-r}.
		\end{split}
	\end{equation*}
	Similarly 
	\begin{equation*}
		\begin{split}
			|\tau_3-  \xi_3^3|=&|\tau_3- (c_3N)^3+  (c_3N)^3- \xi_3^3|\leq  |\tau_2-  \alpha N^3|+|c_3N- \xi_3|\,|(c_3N)^2+(c_3N)\xi_3+ \xi_3^2|\\
			\lesssim &  \frac{1}2N^{-r}+ \,\frac{1}2N^{-(2+r)} N^2
			\lesssim   \,N^{-r}.
		\end{split}
	\end{equation*}
	Thus
	$$
	\langle \tau_j- \alpha \xi_j^3 \rangle \gtrsim N^{-r}, \,\,j=1,2, \quad \langle \tau_3-  \xi_3^3 \rangle \gtrsim N^{-r}, \,\, \langle \tau-  \xi^3 \rangle \gtrsim N^{-r}
	$$
	With these considerations, we get from \eqref{xfail-be2}
	\begin{equation}\label{fail-be81}
		\begin{split}
			\mathcal{B}_s(f,g) 
			&\sim \Big\| N^{\frac12-r (-\frac12+4\epsilon)}\int_{\R^4}\dfrac{ N^{\frac32}\chi_{R_1}(\xi_1, \tau_1)\chi_{R_2}(\xi_2, \tau_2)\chi_{R_3}(\xi_3, \tau_3) }{N^{-3(\frac12+\epsilon)r}}d\xi_1 d\tau_1d\xi_2 d\tau_2\Big\|_{L^2_{\xi\tau}(\R^2)}\\
			&\sim N^{2+2r-r \epsilon} \|\chi_{R_1}*\chi_{R_2}*\chi_{R_3}\|_{L^2(\R^2)}\\
			&= N^{2+2r-r \epsilon}|R_j|^{3-\frac12} = N^{2+2r-r \epsilon}N^{-5-5r}=N^{-3-r(3+ \epsilon)}.
		\end{split}
	\end{equation}
	Now, using  \eqref{xfg-norms} and  \eqref{fail-be81} in \eqref{xfail-be2}, 
	\begin{equation}\label{fail-be91}
		N^{-3-r(3+ \epsilon)}\lesssim N^{-3-3r} \Longleftrightarrow N^{-r\epsilon}\lesssim 1.
	\end{equation}
	
	Since $r<0$, if we choose $N$ large, the estimate \eqref{fail-be91} fails to hold whenever  $\epsilon>0$ and this completes the proof of the proposition.
\end{proof}
%
%
%
%
%

\secao{Failure of Bilinear estimates in Section 3}

In this section we will conclude that the Theorem \ref{loc-sys} is sharp in the sense that we cannot use the approach developed in Section \ref{tri}, to improves the Sobolev indices in Proposition \ref{prop1}.

\begin{proposition}\label{gail-b1}
	Let $\alpha \neq 0,1$. \\
	(a) If $s-k>-\frac12$ or $k<-1/2$ then  the following bilinear estimate
	\begin{equation}\label{fail-be1}
		\|uv\|_{L^2(\R^2)}\lesssim \|u\|_{X^{\alpha}_{k-s-\frac12, \frac12+\epsilon}}\|v\|_{X_{s, \frac12+\epsilon}},
	\end{equation}
	fails to hold.\\
	(b) If  $k-s>-\frac12$ or $s<-1/2$ then  the following bilinear estimate
	\begin{equation}\label{fail-be11}
		\|uv\|_{L^2(\R^2)}\lesssim \|u\|_{X_{s-k-\frac12, \frac12+\epsilon}}\|v\|_{X^{\alpha}_{k, \frac12+\epsilon}},
	\end{equation}
	fails to hold.
\end{proposition}
\begin{proof} As before, we only prove item (a).	Using Plancherel's identity, 
	the estimate \eqref{bil-2} is equivalent to showing that
	\begin{equation}\label{fail-be2}
		\mathcal{B}_s(f,g):=\Big\| \int_{\R^2} \dfrac{\langle \xi_1 \rangle^{s-k+1/2}\widetilde{f}(\xi_2, \tau_2)\widetilde{g}(\xi_1, \tau_1) }{\langle \xi_2\rangle^s\langle \tau_1-\alpha\xi_1^3 \rangle^{b} \langle \tau_2-\xi_2^3 \rangle^{b} } d\xi_1 d\tau_1 \Big\|_{L^2_{\xi\tau}(\R^2)} \lesssim \|f\|_{L^2(\R^2)} \|g\|_{L^2(\R^2)},
	\end{equation}
	where $b=\frac12+\epsilon$. 
	
	Let $\sigma_1=\tau_1-\alpha\xi_1^3$, $\sigma_2=\tau_2-\xi_2^3$, we define the sets
	\begin{equation*}
		A_N=\{(\xi_1, \tau_1) ;\quad |\xi_1-N|< N^{-2}, |\sigma_1|<C_\alpha\} \quad \textrm{and} \quad B_N=\{(\xi_2, \tau_2) ;\quad |\xi_2|< (2N)^{-2}, |\sigma_2|<1\},
	\end{equation*}
	and 
	\begin{equation*}
		R_N=\{(\xi, \tau) ;\quad |\xi-N|< (2N)^{-2}, |\tau-\alpha \xi^3+(1-\alpha)\xi^2N|<1\}.
	\end{equation*}
	Then $R_N-B_N \subset A_N$. In fact, if $(\xi_1, \tau_1)=(\xi, \tau)-(\xi_2, \tau_2)$ with $(\xi, \tau) \in R_N$ and $(\xi_2, \tau_2) \in B_N$, then
	$$
	|\xi_1-N|= |\xi -\xi_2 -N|\leq |\xi  -N|+|\xi_2|< N^{-2},
	$$
	also observe that $\sigma_1+ \sigma_2= \tau -\alpha \xi^3-(1-\alpha)\xi_1^3+3 \alpha \xi \xi_1\xi_2$, thus
	\begin{equation}
		\begin{split}
			|\sigma_1| \leq & |\sigma_2| +|\tau -\alpha \xi^3-(1-\alpha)\xi_1^3+3 \alpha \xi \xi_1\xi_2|\\
			\leq &1+|\tau-\alpha \xi^3+(1-\alpha)\xi^2N|+|1-\alpha|\, |\xi^2N -\xi_1^3|+3|\alpha \xi \xi_1\xi_2|\\
			\leq &2+|1-\alpha|\,\left[ N|(\xi-\xi_1)(\xi+\xi_1)|+ \xi_1^2|N-\xi_1| \right]+3|\alpha \xi \xi_1\xi_2|\\
			\leq &C_\alpha.
		\end{split}
	\end{equation}
	Consider $\widetilde{f}=\chi_{A_N}$, $\widetilde{g}=\chi_{B_N}$. We have $\langle  \xi_1\rangle \sim N$, $\langle  \xi_2\rangle \sim 1$, $\langle \sigma_j \rangle \sim 1$, $j=1,2$. From \eqref{fail-be2} and Lemma \ref{lemleandro}, we obtain
	\begin{equation*}
		\begin{split}
			N^{s-k+1/2} \|R_N\|_{L^2(\R^2)} \|\chi_{B_N}\|_{L^1(\R^2)}\leq N^{s-k+1/2} \|\chi_{A_N} \ast \chi_{B_N}\|_{L^2(\R^2)}\leq
			\|\chi_{A_N} \|_{L^2(\R^2)}\| \chi_{B_N}\|_{L^2(\R^2)}\end{split}
	\end{equation*}
	
	which implies
	$$
	N^{s-k+1/2} N^{-1} N^{-2} \lesssim N^{-2}
	$$
	and the last inequality is false if $s-k>1/2$.
	
	Analogously, if we consider the sets
	\begin{equation*}
		A_N=\{(\xi_1, \tau_1) ;\ |\xi_1-N|< N^{-2}, |\tau_1- \alpha N^{3}|<2\},
	\end{equation*}
	\begin{equation*}\ B_N=\{(\xi_2, \tau_2) ;\ |\xi_2+N|< (2N)^{-2}, |\tau_2+N^{3}|<1\}
	\end{equation*}
	and 
	\begin{equation*}
		R_N=\{(\xi, \tau) ;\quad |\xi|< (2N)^{-2}, |\tau-(\alpha-1)N^{3}|<1\},
	\end{equation*}
	then $R_N-B_N \subset A_N$ and 
	$$
	|\tau_1-\alpha \xi_1^{3}|\leq |\tau_1-\alpha N^{3}+\alpha (N^{3}-\xi_1^{3})| \lesssim 2+|\alpha|,
	$$
	$$
	|\tau_2-\xi_2^{3}|\leq |\tau_2+N^{3}- (N^{3}+\xi_2^{3})| \lesssim 1.
	$$
	Noting that 
	$$
	\langle  \xi_1\rangle \sim N, \,\, \langle  \xi_2\rangle \sim N, \,\,\langle \sigma_j \rangle \sim 1, \,\,j=1,2, 
	$$
	from \eqref{fail-be2} we obtain
	\begin{equation*}
		\begin{split}
			N^{-k+1/2} \|R_N\|_{L^2(\R^2)} \|\chi_{B_N}\|_{L^1(\R^2)}\leq N^{-k+1/2} \|\chi_{A_N} \ast \chi_{B_N}\|_{L^2(\R^2)}\leq
			\|\chi_{A_N} \|_{L^2(\R^2)}\| \chi_{B_N}\|_{L^2(\R^2)},\end{split}
	\end{equation*}
	which implies
	$$
	N^{-k+1/2} N^{-1} N^{-2} \lesssim N^{-2}
	$$
	and the last inequality is false if $k<-1/2$. 
	
\end{proof}

At the endpoint we have the following result.

\begin{proposition}\label{gail-b3}
	Let $\alpha \neq 0, 1$, then  the following bilinear estimate
	\begin{equation}\label{fail-be3}
		\|uv\|_{L^2(\R^2)} \lesssim \|u\|_{X_{-\frac12, \frac12-2\epsilon}^{\alpha}}\|v\|_{X_{-\frac12, \frac12+\epsilon}}
	\end{equation}
	fails to hold for all $\epsilon >0$.
\end{proposition}
\begin{proof}
	Using Plancherel's identity, 
	the estimate \eqref{fail-be1} is equivalent to showing that
	\begin{equation}\label{fail-be2x}
		\mathcal{B}_s(f,g):=\Big\| \int_{\R^2} \dfrac{\langle \xi_2 \rangle^{1/2}\langle \xi_1 \rangle^{1/2}\tilde{f}(\xi_2, \tau_2)\tilde{g}(\xi_1, \tau_1) }{\langle \tau_1- \xi_1^3 \rangle^{\frac12+\epsilon} \langle \tau_2-\alpha\xi_2^3 \rangle^{\frac12-2\epsilon} } d\xi_1 d\tau_1 \Big\|_{L^2_{\xi\tau}(\R^2)} \le \|f\|_{L^2(\R^2)} \|g\|_{L^2(\R^2)},
	\end{equation}
	where 
	$\widetilde{f}(\xi,\tau)=\langle\xi\rangle^{-\frac12}\langle\tau-\alpha\xi^3\rangle^{\frac12-2\epsilon}\widetilde{u}(\xi,\tau)$,  $\widetilde{g}(\xi,\tau)=\langle\xi\rangle^{-\frac12}\langle\tau-\xi^3\rangle^{\frac12+\epsilon} \widetilde{v}(\xi,\tau)$, $\xi_2=\xi-\xi_1$ and $\tau_2=\tau-\tau_1$.

	We will construct functions $f$ and $g$ for which the estimate \eqref{fail-be2x} fails to hold for all $\epsilon>0$. 
	
	Consider two rectangles $R_1$ and  $R_2$  centered respectively at $(N, \alpha N^3)$, and  $(N, N^3)$, and each with dimension $N^{-(2+r)}\times N^{-r}$, where $-2<r<0$. Now, let $f$ and $g$ defined, via their Fourier transform, by
	$\widetilde{f} = \chi_{R_1}$ and  $ \widetilde{g}= \chi_{R_2}$.	It is easy to see that 
	\begin{equation}\label{fg-norms}
		\|f\|_{L^2(\R^2)} =\|g\|_{L^2(\R^2)}=N^{-(1+r)}.
	\end{equation}
	Also, 
	\begin{equation}\label{fail-be5}
		| \xi_1-N|\leq \frac{1}2N^{-(2+r)}, \qquad |\tau_1- N^3|\leq \frac{1}2N^{-r}
	\end{equation}
	and
	\begin{equation}\label{fail-be6}
		|\xi_2-N|\leq \frac{1}2N^{-(2+r)}, \qquad |\tau_2-\alpha N^3|\leq \frac{1}2N^{-r}.
	\end{equation}
	
	We have $|\xi_j|\sim N$.
	Also, we have   that 
	\begin{equation*}
		\begin{split}
			|\tau_1- \xi_1^3|=&|\tau_1- N^3+ N^3- \xi_1^3|
			\leq  |\tau_1-  N^3|+\,|N^3- \xi_1^3|\\
			\leq &  \frac{1}2N^{-r}+ \frac{1}2N^{-(2+r)}|N^2+N\xi_1+ \xi_1^2|\\
			\lesssim &  \,N^{-r}.
		\end{split}
	\end{equation*}
	Similarly, 
	\begin{equation*}
		\begin{split}
			|\tau_2- \alpha \xi_2^3|=&|\tau_2- \alpha N^3+ \alpha N^3-\alpha \xi_2^3|
			\leq  |\tau_2-  \alpha N^3|+|\alpha|\,|N^3- \xi_2^3|\\
			\leq &  \frac{1}2N^{-r}+ |\alpha|\,\frac{1}2N^{-(2+r)}|N^2+N\xi_2+ \xi_2^2|\\
			\lesssim &  \,N^{-r}.
		\end{split}
	\end{equation*}
	
	With these considerations, we get from \eqref{fail-be2x} and Lemma \ref{conv-lema2}
	\begin{equation}\label{fail-be8}
		\begin{split}
			\mathcal{B}_s(f,g) 
			&\gtrsim \Big\| N^{1+r-r \epsilon}\int_{\R^2} \chi_{R_1}(\xi_1, \tau_1)\chi_{R_2}(\xi_2, \tau_2) d\xi_1 d\tau_1 \Big\|_{L^2_{\xi\tau}(\R^2)}\\
			&\sim N^{1+r-r \epsilon} \|\chi_{R_1}*\chi_{R_2}\|_{L^2(\R^2)}\\
			&= N^{1+r-r \epsilon}|R_j|^{2-\frac12} = N^{1+r-r \epsilon}N^{-3-3r}=N^{-2-r(2+ \epsilon)}.
		\end{split}
	\end{equation}
	
	Now, using  \eqref{fg-norms} and  \eqref{fail-be8} in \eqref{fail-be2}, 
	\begin{equation}\label{fail-be9}
		N^{-2-r(2+ \epsilon)} \lesssim N^{-2-2r} \Longleftrightarrow N^{-r\epsilon}\lesssim 1.
	\end{equation}
	
	Since $r<0$, if we choose $N$ large, the estimate \eqref{fail-be9} fails to hold whenever  $\epsilon>0$ and this completes the proof of the proposition.
\end{proof}

%

\end{document}